\documentclass[12pt]{article}%
\usepackage{amsmath}
\usepackage{graphicx}%
\usepackage{amsfonts}%
\usepackage{amssymb}
\newtheorem{theorem}{Theorem}

\newtheorem{corollary}[theorem]{Corollary}

\newtheorem{definition}[theorem]{Definition}

\newtheorem{lemma}[theorem]{Lemma}

\newtheorem{proposition}[theorem]{Proposition}
\newtheorem{remark}[theorem]{Remark}

\newenvironment{proof}[1][Proof]{\textbf{#1.} }{\ \rule{0.5em}{0.5em}}

\begin{document}

\title{MARKOV LOOPS AND RENORMALIZATION }
\author{Yves Le Jan\\D\'{e}partement de Math\'{e}matiques\\Universit\'{e} Paris Sud 11\\yves.lejan@math.u-psud.fr}
\maketitle

\begin{abstract}
We study the Poissonnian ensembles of Markov loops and the associated
renormalized self intersection local times.

\end{abstract}

\bigskip

\baselineskip                         =16pt

\section{\bigskip Introduction}

\bigskip

The purpose of this paper is to explore some simple relations between
Markovian path and loop measures, spanning trees, determinants, and Markov
fields such as the free field. The main emphasis is put on the study of
occupation fields defined by Poissonian ensembles of Markov loops. These were
defined in \cite{LW} for planar Brownian motion in relation with SLE processes
and in \cite{LT} for simple random walks. They appeared informally already in
\cite{Symanz}. For half integral values $\frac{k}{2}$ of the intensity
parameter $\alpha$, these occupation fields can be identified with the sum of
squares of $k$ copies of the associated free field (i.e. the Gaussian field
whose covariance is given by the Green function)$.$ This is related to
Dynkin's isomorphism (cf \cite{Dy}, \cite{MR}, \cite{LJ1}). We first present
the results in the elementary framework of symmetric Markov chains on a finite
space, proving also en passant several interesting results such as the
relation between loop ensembles and spanning trees. Then we show some results
can be extended to more general Markov processes. There are no essential
difficulties when points are not polar but other cases are more problematic.
As for the square of the free field, cases for which the Green function is
Hilbert Schmidt such as two and three dimensional Brownian motion can be dealt
with through appropriate renormalization.

We can show that the renormalised powers of the occupation field (i.e. the
self intersection local times of the loop ensemble) converge in the two
dimensional case and that they can be identified with higher even Wick powers
of the free field when $\alpha$ is a half integer.

\section{Symmetric Markov processes on finite spaces}

Notations: Functions and measures on finite (or countable) spaces are often
denoted as vectors and covectors.

The multiplication operator defined by a function $f$\ acting on functions or
on measures is in general simply denoted by $f$, but sometimes it will be
denoted $M_{f}$. The function obtained as the density of a measure $\mu$ with
respect to some other measure $\nu$ is simply denoted $\frac{\mu}{\nu}$.

Our basic object will be a finite space $X$\ and a set of \ non negative
conductances $C_{x,y}=C_{y,x}$, indexed by pairs of distinct points of $X$.

We say $\ \{x,y\}$ is a link or an edge iff $C_{x,y}>0$ and an oriented edge
$(x,y)$\ is defined by the choice of an ordering in an edge$.$ We set
$-(x,y)=(y,x)$ and if $e=(x,y)$, we denote it also $(e^{-},e^{+})$.

The points of $X$ together with the set of non oriented edges $E$\ define a
graph.$(X,E)$. \textsl{We assume it is connected}. The set of oriented edges
is denoted $E^{o}$.

An important example is the case in which conductances are equal to zero or
one. Then the conductance matrix is the adjacency matrix of the graph:
$C_{x,y}=1_{\{x,y\}\in E}$

\subsection{Energy}

Let us consider a nonnegative function $\kappa$ on $X$. Set $\lambda
_{x}=\kappa_{x}+\sum_{y}C_{x,y}$\ $P_{y}^{x}=\frac{C_{x,y}}{\lambda_{x}}$. $P$
is a $\lambda$-symmetric (sub) stochastic transition matrix: $\lambda_{x}%
P_{y}^{x}=\lambda_{y}P_{x}^{y}$ with $P_{x}^{x}=0$\ for all $x$ in $X$ and it
defines a symmetric irreducible Markov chain $\xi_{n}$.

We can define above it a $\lambda$-symmetric irreducible Markov chain in
continuous time $x_{t}$, with exponential holding times,of parameter $1$. We
have $x_{t}=\xi_{N_{t}}$, where $N_{t}$ denotes a Poisson process of intensity
$1$.The infinitesimal generator writes $L_{y}^{x}=P_{y}^{x}-\delta_{y}^{x}$.

We denote by $P_{t}$ its (sub) Markovian semigroup $\exp(Lt)=\sum\frac{t^{k}%
}{k!}L^{k}$. $L$ and $P_{t}$ are $\lambda$-symmetric.

We will consider the Markov chain associated with $C,\kappa$, sometimes in
discrete time, sometimes in continuous time (with exponential holding times).

Recall that for any complex function $z^{x},x\in X$, the ''energy''%
\[
e(z)=\left\langle -Lz,\overline{z}\right\rangle _{\lambda}=\sum_{x\in
X}-(Lz)^{x}\overline{z}^{x}\lambda_{x}%
\]
is nonnegative as it can be written%
\[
e(z)=\frac{1}{2}\sum_{x,y}C_{x,y}(z^{x}-z^{y})(\overline{z}^{x}-\overline
{z}^{y})+\sum_{x}\kappa_{x}z^{x}\overline{z}^{x}=\sum_{x}\lambda_{x}%
z^{x}\overline{z}^{x}-\sum_{x,y}C_{x,y}z^{x}\overline{z}^{y}%
\]
The Dirichlet space (\cite{Fukutak}) is the space of\ real functions equipped
with the energy scalar product defined by polarization of $e$.

Note that the non negative symmetric ''conductance matrix'' $C$ and the non
negative equilibrium or ''killing'' (or ''equilibrium'') measure $\kappa$ are
the free parameters of the model.

We have a dichotomy between:

\begin{itemize}
\item[-] the recurrent case where $0$ is the lowest eigenvalue of $-L$, and
the corresponding eigenspace is formed by constants. Equivalently, $P1=1$ and
$\kappa$ vanishes.

\item[-] the transient case where the lowest eigenvalue is positive which
means there is a ''Poincar\'{e} inequality'': For some positive $\varepsilon$,
the energy $e(f,f)$ dominates $\varepsilon\left\langle f,f\right\rangle
_{\lambda}$ for all $f$. Equivalently, $\kappa$ does not vanish.
\end{itemize}

We will now work in the transient case. We denote by $V$ the associated
potential operator $(-L)^{-1}=\int_{0}^{\infty}P_{t}dt$. It can be expressed
in terms of the spectral resolution of $L$.

We denote by $G$ the Green function defined on $X^{2}$ as $G^{x,y}%
=\frac{V_{y}^{x}}{\lambda_{y}}=\frac{1}{\lambda_{y}}[(I-P)^{-1}]_{y}^{x}$ i.e.
$G=(M_{\lambda}-C)^{-1}$. It induces a linear bijection from measures into
functions. We set $(G\mu)^{x}=\sum_{y}G^{x,y}\mu_{y}$

Note that $e(f,G\mu)=\left\langle f,\mu\right\rangle $ (i.e. $\sum_{x}f^{x}%
\mu_{x}$) for all function $f$ and measure $\mu$. In particular $G\kappa=1$ as
$\ e(1,f)=\sum f^{x}\kappa_{x}=\left\langle f,1\right\rangle _{\kappa}$.

See (\cite{Fukutak}) for a development of this theory in a more general setting.

In the recurrent case, the potential operator $V$ operates on the space
$\lambda^{\perp}$ of functions $f$ such that $\left\langle f,1\right\rangle
_{\lambda}=0$ as the inverse of the restriction of $I-P$ to $\lambda^{\perp}$.
The Green operator $G$ maps the space of measures of total charge zero onto
$\lambda^{\perp}$. Setting for any signed measure $\nu$ of total charge zero
$G\nu=V\frac{\nu}{\lambda}$. we have for any function $f$, $\left\langle
\nu,f\right\rangle =e(G\nu,f)$ (as $e(G\nu,1)=0$) and in particular$\ f^{x}%
-f^{y}=e(G(\delta_{x}-\delta_{y}),f)$.

\subsection{Feynman-Kac formula}

For the continuous time Markov chain $x_{t}$ (with exponential holding times)
and $k(x)$ any non negative function, we have the Feynman Kac formula:
\[
\mathbb{E}_{x}(e^{-\int_{0}^{t}k(x_{s})ds}1_{\{x_{t}=y\}})=[\exp
(t(L-M_{k})]_{y}^{x}.
\]

For any nonnegative measure $\chi$, set $V_{\chi}=(-L+M_{_{\frac{\chi}%
{\lambda}}})^{-1}$ and $G_{\chi}=V_{\chi}M_{\frac{1}{\lambda}}=(M_{\lambda
}+M_{\chi}-C)^{-1}$. It is a symmetric nonnegative function on $X\times X$.
$G_{0}$ is the Green function $G$, and $G_{\chi}$ can be viewed as the Green
function of the energy form $e_{\chi}=e+\left\|  \quad\right\|  _{L^{2}(\chi
)}^{2}$.

Note that $e_{\chi}$ has the same conductances $C$ as $e,$ but $\chi$ is added
to the killing measure. Note also that $V_{\chi}$ is not the potential of the
Markov chain associated with $e_{\chi}$ when one takes exponential holding
times of parameter $1$ but the Green function is intrinsic i.e. invariant
under a change of time scale. Still, we have by Feynman Kac formula
\[
\int_{0}^{\infty}\mathbb{E}_{x}(e^{-\int_{0}^{t}\frac{\chi}{\lambda}(x_{s}%
)ds}1_{\{x_{t}=y\}})dt=[V_{\chi}]_{y}^{x}.
\]
We have also the ''resolvent'' equation $V-V_{\chi}=VM_{\frac{\chi}{\lambda}%
}V_{\chi}=V_{\chi}M_{\frac{\chi}{\lambda}}V$. Then,
\[
G-G_{\chi}=GM_{\chi}G_{\chi}=G_{\chi}M_{\chi}G
\]

\subsection{Countable spaces}

The assumption of finiteness of $X$ can be relaxed. On countable spaces, the
previous results extend easily when under spectral gap conditions. In the
transient case, the Dirichlet space $\mathbb{H}$\ is the space of all
functions $f$\ with finite energy $e(f)$ which are limits in energy norm of
functions with finite support. The energy of a measure is defined as
$\sup_{f\in\mathbb{H}}\frac{\mu(f)^{2}}{e(f)}$. It includes Dirac measures.
The potential $G\mu$ is well defined in $\mathbb{H}$\ for all finite energy
measures $\mu$, by the identity $e(f,G\mu)=\left\langle f,\mu\right\rangle $,
valid for all $f$ in the Dirichlet space.

Most important cases are the non ramified covering of finite graphs.

\section{Loop measures}

\subsection{A measure on based loops}

We denote $\mathbb{P}^{x}$\ the family of probability laws on piecewise
constant paths defined by $P_{t}$.%
\[
\mathbb{P}^{x}(\gamma(t_{1})=x_{1},...,\gamma(t_{h})=x_{h})=P_{t_{1}}%
(x,x_{1})P_{t_{2}-t_{1}}(x_{1},x_{2})\ldots P_{t_{h}-t_{h-1}}(x_{h-1},x_{h})
\]
Denoting by $p(\gamma)$ the number of jumps and $T_{i}$ the jump times, we
have:%
\begin{multline*}
\mathbb{P}_{x}(p(\gamma)=k,\gamma_{T_{1}}=x_{1},...,\gamma_{T_{k-1}}%
=x_{k-1},T_{1}\in dt_{1},...,T_{k}\in dt_{k})\\
=\frac{C_{x,x_{2}}...C_{x_{k-1},x_{k}}\kappa_{x_{k}}}{\lambda_{x}%
\lambda_{x_{2}}...\lambda_{x_{k}}}1_{\{0<t_{1}<...<t_{k}\}}e^{-t_{k}}%
dt_{1}...dt_{k}%
\end{multline*}
For any integer $p>2$, let us define a based loop with $p$ points in $X$ as a
couple $l=(\xi,\tau)=((\xi_{m},1\leq m\leq p),(\tau_{m},1\leq m\leq
p+1),\mathbb{)}$ in $X^{p}\times\mathbb{R}_{+}^{p+1}$, and set $\xi_{1}%
=\xi_{p+1}$ (equivalently, we can parametrize the the discrete based loop by
$\mathbb{Z}/p\mathbb{Z}$). The integer $p$ represents the number of points in
the discrete based loop $\xi=(\xi_{1},...\xi_{p(\xi)})$ and will be denoted
$p(\xi)$. Note two time parameters are attached to the base point since the
based loops do not in general end or start with a jump.

Based loops with one point $(p=1)$\ are simply given by a pair $(\xi,\tau
)$\ in $X\times\mathbb{R}_{+}$.

Based loops have a natural time parametrization $l(t)$ and a time period
$T(\xi)=\sum_{i=1}^{p(\xi)+1}\tau_{i}$. If we denote $\sum_{i=1}^{m}\tau_{i}%
$\ by $T_{m}$: $l(t)=\xi_{m-1}$ on $[T_{m-1},T_{m})$ (with by
convention\ $T_{0}=0$ and $\xi_{0}=\xi_{p}$).

A $\sigma$-finite measure $\mu$ is defined on based loops by%
\[
\mu=\sum_{x\in X}\int_{0}^{\infty}\frac{1}{t}\mathbb{P}_{t}^{x,x}\lambda
_{x}dt
\]
where $\mathbb{P}_{t}^{x,y}$ denotes the (non normalized) ''law'' of a path
from $x$ to $y$\ of duration $t$ : If $t_{1}<t_{2}...<t_{h}<t$,%
\[
\mathbb{P}_{t}^{x,y}(l(t_{1})=x_{1},...,l(t_{h})=x_{h})=[P_{t_{1}}]_{x_{1}%
}^{x}[P_{t_{2}-t_{1}}]_{x_{2}}^{x_{1}}...[P_{t-t_{h}}]_{y}^{x_{h}}%
\frac{1}{\lambda_{y}}%
\]
Its mass is $p_{t}^{x,y}=\frac{[P_{t}]_{y}^{x}}{\lambda_{y}}$. And for any
measuable set $A$\ of piecewise constant paths indexed by $[0\;t]$, we can
also write
\[
\mathbb{P}_{t}^{x,y}(A)=\mathbb{P}_{x}(A\cap\{x_{t}=y\})\frac{1}{\lambda_{y}}%
\]

From the first expression, we see that by definition of $\mu$, if $t_{1}%
<t_{2}...<t_{h}<t$,%
\begin{equation}
\mu(l(t_{1})=x_{1},...,l(t_{h})=x_{h},T\in dt)=[P_{t_{1}+t-t_{h}}]_{x_{1}}%
^{x}[P_{t_{2}-t_{1}}]_{x_{2}}^{x_{1}}...[P_{t_{h}-t_{h-1}}]_{x_{h}}^{x_{h-1}%
}\frac{1}{t}dt \label{mu1}%
\end{equation}

Note also that for $k>1$, using the second expression of $\mathbb{P}_{t}%
^{x,x}$ and the fact that conditionally to $N_{t}=k$, the jump times are
distributed like an increasingly reordered $k-$uniform sample of $[0\;t]$%
\begin{align*}
\lambda_{x}\mathbb{P}_{t}^{x,x}(p  &  =k,\xi_{2}=x_{2},...,\xi_{_{k}}%
=x_{k},T_{1}\in dt_{1},...,T_{k}\in dt_{k})\\
&  =P_{x_{2}}^{x}P_{x_{3}}^{x_{2}}...P_{x}^{x_{k}}1_{\{0<t_{1}<...t_{k}%
<t\}}e^{-t}dt_{1}...dt_{k}%
\end{align*}
Therefore
\begin{align}
\mu(p  &  =k,\xi_{1}=x_{1},..,\xi_{k}=x_{k},T_{1}\in dt_{1},..,T_{k}\in
dt_{k},T\in dt)\label{mu3}\\
&  =P_{x_{2}}^{x_{1}}..P_{x_{1}}^{x_{k}}\frac{1_{\{0<t_{1}<...<t_{k}<t\}}}%
{t}e^{-t}dt_{1}...dt_{k}dt
\end{align}

for $k>1$.

Moreover, for one point-loops, $\mu\{p(\xi)=1,\xi_{1}=x_{1},\tau_{1}\in
dt\}=\frac{e^{-t}}{t}dt$

\subsection{First properties}

Note that the loop measure is invariant under time reversal.

If $D$ is a subset of $X$, the restriction of $\mu$ to loops contained in $D$,
denoted $\mu^{D}$ is clearly the loop measure induced by the Markov chain
killed at the exit of $D$. This can be called the \textsl{restriction
property}.

Let us recall that this killed Markov chain is defined by the restriction of
$\lambda$ to $D$ and the restriction $P^{D}$ of $P$ to $D^{2}$ (or
equivalently by the restriction $e_{D}$ of the Dirichlet norm $e$ to functions
vanishing outside $D$).

As $\int\frac{t^{k-1}}{k!}e^{-t}dt=\frac{1}{k}$, it follows from (\ref{mu3})
that for $k>1$, on based loops,%
\begin{equation}
\mu(p(\xi)=k,\xi_{1}=x_{1},...,\xi_{k}=x_{k})=\frac{1}{k}P_{x_{2}}^{x_{1}%
}...P_{x_{1}}^{x_{k}} \label{de}%
\end{equation}
In particular, we obtain that, for $k\geq2$
\[
\mu(p=k)=\frac{1}{k}Tr(P^{k})
\]
and therefore, as $Tr(P)=0$,
\[
\mu(p>1)=\sum_{2}^{\infty}\frac{1}{k}Tr(P^{k})=-\log(\det(I-P))=\log
(\det(G)\prod_{x}\lambda_{x})
\]
since (denoting $M_{\lambda}$ the diagonal matrix with entries $\lambda_{x}$),
we have
\[
{\det(I-P)=\frac{\det(M_{\lambda}-C)}{\det(M_{\lambda})}}%
\]
Moreover
\[
\int p(l)1_{\{p>1\}}\mu(dl)=\sum_{2}^{\infty}Tr(P^{k})=Tr((I-P)^{-1}P)=Tr(GC)
\]

\subsection{Loops and pointed loops}
It is clear on formula \ref{mu1} that $\mu$ is invariant under the time shift
that acts naturally on based loops.\bigskip

A loop is defined as an equivalence class of based loops for this shift.
Therefore, $\mu$ induces a\textsl{ measure\ on loops also denoted by }$\mu$.

A loop is defined by the discrete loop $\xi^{%
{{}^\circ}%
}$\ formed by the $\xi_{i}$ in circular order, (i.e. up to translation) and
the associated scaled holding times. We clearly have:%
\[
\mu(\xi^{%
{{}^\circ}%
}=(x_{1},x_{2},...,x_{k})^{%
{{}^\circ}%
})=P_{x_{2}}^{x_{1}}...P_{x_{1}}^{x_{k}}%
\]

However, loops are not easy to parametrize, that is why we will work mostly
with based loops or \textsl{pointed loops}$.$ These are defined as based loops
ending with a jump, or as loops with a starting point. They can be
parametrized by a based discrete loop and by the holding times at each point.
Calculations are easier if we work with based or pointed loops, even though we
will deal only with functions independent of the base point.\bigskip

The parameters of the pointed loop naturally associated with a based loop are
$\xi_{1},...,\xi_{p}$ and
\[
\tau_{1}+\tau_{p+1}{=\tau_{1}^{\ast},\tau_{i}=\tau_{i}^{\ast},\ 2\leq i\leq p}%
\]
An elementary change of variables, shows the expression of $\mu$ on pointed
loops writes:%

\begin{equation}
\mu(p=k,\xi_{i}=x_{i},{\tau_{i}^{\ast}\in dt}_{i})=P_{x_{2}}^{x_{1}%
}...P_{x_{1}}^{x_{k}}\frac{t_{1}}{\sum t_{i}}e^{-\sum t_{i}}dt_{1}...dt_{k}
\label{dd}%
\end{equation}
Trivial ($p=1$) pointed loops and trivial based loops coincide.

Note that loop functionals can be written%
\[
\Phi(l%
{{}^\circ}%
)=\sum1_{\{p=k\}}\Phi_{k}((\xi_{i},\tau_{i}^{\ast}),i=1,...k)
\]
with $\Phi_{k}$ invariant under circular permutation of the variables
$(\xi_{i},\tau_{i}^{\ast})$.

Then, for non negative $\Phi_{k}$
\[
\int\Phi_{k}(l^{%
{{}^\circ}%
})\mu(dl)=\int\sum_{x_{i}}\Phi_{k}(x_{i},t_{i})P_{x_{2}}^{x_{1}}...P_{x_{1}%
}^{x_{k}}e^{-\sum t_{i}}\frac{t_{1}}{\sum t_{i}}dt_{1}...dt_{k}%
\]
and by invariance under circular permutation, the term $t_{1}$ can be replaced
by any $t_{i}$. Therefore, adding up and dividing by $k$, we get that%

\[
\int\Phi_{k}(l^{%
{{}^\circ}%
})\mu(dl)=\int\frac{1}{k}\sum_{x_{i}}\Phi_{k}(x_{i},t_{i})P_{x_{2}}^{x_{1}%
}...P_{x_{1}}^{x_{k}}e^{-\sum t_{i}}dt_{1}...dt_{k}%
\]

The expression on the right side, applied to any pointed loop functional
defines a different measure on pointed loops, we will denote by $\mu^{\ast}$.
It induces the same measure as $\mu$ on loops.

\bigskip We see on this expression that conditionally to the discrete loop,
the holding times of the loop are independent exponential variables.%

\begin{equation}
\mu^{\ast}(p=k,\xi_{i}=x_{i},{\tau_{i}^{\ast}\in dt}_{i})=\frac{1}{k}%
\prod_{i\in\mathbb{Z}/p\mathbb{Z}}C_{\xi_{i},\xi_{i+1}}e^{-t_{i}}dt_{i}
\label{ddd}%
\end{equation}

Conditionally to $p(\xi)=k,$ $T$ is a gamma variable of density $\frac{t^{k-1}%
}{(k-1)!}e^{-t}$ on $\mathbb{R}_{+}$ and ${(\frac{{\tau_{i}^{\ast}}}%
{T},\ 1\leq i\leq k)}$ an independent ordered $k$-sample of the uniform
distribution on $(0,T)$ (whence the factor $\frac{1}{t}$). Both are
independent, conditionally to $p$ of the discrete loop. We see that $\mu,$ on
based loops, is obtained from $\mu$ on the loops by choosing the based point
uniformly. On the other hand, it induces a choice of $\xi_{1}$ biased by the
size of the $\tau_{i}^{\ast}$'s, different of $\mu^{\ast}$ (whence the factor
$\frac{1}{k}$. But we will consider only loop functionals.

It will be convenient to rescale the holding time at each $\xi_{i}$\ by
$\lambda_{\xi_{i}}$ and set $\widehat{\tau}_{i}=\frac{\tau_{i}^{\ast}}%
{\lambda_{\xi_{i}}}$.

The discrete part of the loop is the most important, though we will see that
to establish a connection with Gaussian fields it is necessary to consider
occupation times. The simplest variables are the number of jumps from $x$ to
$y$, defined for every oriented edge $(x,y)$%
\[
N_{x,y}=\#\{i:\xi_{i}=x,\xi_{i+1}=y\}
\]
(recall the convention $\xi_{p+1}=\xi_{1})$ and%
\[
N_{x}=\sum_{y}N_{x,y}%
\]
Note that $N_{x}=\#\{i\geq1:\xi_{i}=x\}$\ except for trivial one point loops
for which it vanishes.\newline Then, the measure on pointed loops (\ref{dd})
can be rewritten as:%
\begin{align}
\mu^{\ast}(p  &  =1,\xi=x,\widehat{\tau}\in dt)=e^{-\lambda_{x}t}\frac{dt}%
{t}\text{ and }\label{d}\\
\mu^{\ast}(p  &  =k,\xi_{i}=x_{i},\widehat{\tau}_{i}\in dt_{i})=\frac{1}%
{k}\prod_{x,y}C_{x,y}^{N_{x,y}}\prod_{x}\lambda_{x}^{-N_{x}}\prod
_{i\in\mathbb{Z}/p\mathbb{Z}}\lambda_{\xi_{i}}e^{-\lambda_{\xi_{i}}t_{i}%
}dt_{i}%
\end{align}

Another \textsl{bridge measure }$\mu^{x,y}$ can be defined on paths $\gamma$
from $x$ to $y$: $\mu^{x,y}(d\gamma)=\int_{0}^{\infty}\mathbb{P}_{t}%
^{x,y}(d\gamma)dt$.\newline Note that the mass of $\mu^{x,y}$ is $G^{x,y}$. We
also have, with similar notations as the one defined for loops, $p$ denoting
the number of jumps%
\begin{multline*}
\mu^{x,y}(p(\gamma)=k,\gamma_{T_{1}}=x_{1},...,\gamma_{T_{k-1}}=x_{k-1}%
,T_{1}\in dt_{1},...,T_{k-1}\in dt_{k-1},T\in dt)\\
=\frac{C_{x,x_{2}}C_{x_{2},x_{3}}...C_{x_{k-1},y}}{\lambda_{x}\lambda_{x_{2}%
}...\lambda_{y}}1_{\{0<t_{1}<...<t_{k}<t\}}e^{-t}dt_{1}...dt_{k}dt
\end{multline*}

\subsection{Occupation field}

To each loop $l^{%
{{}^\circ}%
}$ we associate local times, i.e. an occupation field $\{\widehat{l_{x}},x\in
X\}$ defined by
\[
\widehat{l}^{x}=\int_{0}^{T(l)}1_{\{\xi(s)=x\}}\frac{1}{\lambda_{\xi(s)}%
}ds=\sum_{i=1}^{p(l)}1_{\{\xi_{i}=x\}}\widehat{\tau_{i}}%
\]
for any representative $l=(\xi_{i},\tau_{i}^{\ast})$ of $l^{\circ}$.

For a path $\gamma$, $\widehat{\gamma}$ is defined in the same way.\newline
Note that%
\begin{equation}
\mu((1-e^{-\alpha\widehat{l}^{x}})1_{\{p=1\}})=\int_{0}^{\infty}%
e^{-t}(1-e^{-\frac{\alpha}{\lambda_{x}}t})\frac{dt}{t}=\log(1+\frac{\alpha
}{\lambda_{x}}) \label{triv}%
\end{equation}
(by expanding $1-e^{-\frac{\alpha}{\lambda_{x}}t}$ before the integration,
assuming first $\alpha$ small and then by analyticity of both members, or more
elegantly, noticing that $\int_{a}^{b}(e^{-cx}-e^{-dx})\frac{dx}{x}$ is
symmetric in $(a,b)$\ and $(c,d)$).\newline In particular, $\mu(\widehat
{l}^{x}1_{\{p=1\}})=\frac{1}{\lambda_{x}}$.

From formula \ref{dd}, we get easily that the joint conditional distribution
of $(\widehat{l}^{x},\ x\in X)$ given $(N_{x},\ x\in X)$ is a product of gamma
distributions. In particular, from the expression of the moments of a gamma
distribution, wee get that for any function $\Phi$ of the discrete loop and
$k\geq1$,
\[
\mu((\widehat{l}^{x})^{k}1_{\{p>1\}}\Phi)=\lambda_{x}^{-k}\mu((N_{x}%
+k-1)...(N_{x}+1)N_{x}\Phi)
\]
In particular, $\mu(\widehat{l}^{x})=\frac{1}{\lambda_{x}}[\mu(N_{x}%
)+1]=G^{x,x}$.

Note that functions of $\widehat{l}$ are not the only functions naturally
defined on the loops. Other such variables of interest are, for $n\geq2$, the
multiple local times, defined as follows:%
\[
\widehat{l}^{x_{1},...,x_{n}}=\sum_{j=0}^{n-1}\int_{0<t_{1}<...<t_{n}%
<T}1_{\{\xi(t_{1})=x_{1+j},....\xi(t_{n-j})=x_{n},...\xi(t_{n})=x_{j}\}}%
\prod\frac{1}{\lambda_{x_{i}}}dt_{i}%
\]
It is easy to check that, when the points $x_{i}$ are distinct,
\begin{equation}
\widehat{l}^{x_{1},...,x_{n}}=\sum_{j=0}^{n-1}\sum_{1\leq i_{1}<..<i_{n}\leq
p(l)}\prod_{l=1}^{n}1_{\{\xi_{i_{l}}=x_{l+j}\}}\widehat{\tau_{i_{l}}}.
\label{tamul}%
\end{equation}
Note that in general $\widehat{l}^{x_{1},...,x_{k}}$ cannot be expressed in
terms of $\widehat{l}$.

If $x_{1}=x_{2}=\ldots=x_{n}$, $\widehat{l}^{x_{1},...,x_{n}}=\frac{1}%
{(n-1)!}[\widehat{l}^{x}]^{n}$. It can be viewed as a $n$-th self intersection
local time.

One can deduce from the defintions of $\mu$\ the following:

\begin{proposition}
$\mu(\widehat{l}^{x_{1},...,x_{n}})=G^{x_{1},x_{2}}G^{x_{2},x_{3}}%
...G^{x_{n},x_{1}}$
\end{proposition}

\begin{proof}
Let us denote $\frac{1}{\lambda_{y}}$ $[P_{t}]_{y}^{x}$ by $p_{t}^{x,y}$ or
$p_{t}(x,y)$.\ From the definition of $\widehat{l}^{x_{1},...,x_{n}}$ and
$\mu$, $\mu(\widehat{l}^{x_{1},...,x_{n}})$ equals:
\[
\sum_{x}\lambda_{x}\sum_{j=0}^{n-1}\int\int_{\{0<t_{1}...<t_{n}<t\}}%
\frac{1}{t}p_{t_{1}}(x,x_{1+j})\ldots p_{t-t_{n}}(x_{n+j},x)\prod dt_{i}dt
\]
where sums of indices $k+j$ are computed $\operatorname{mod}(n)$. By the
semigroup property, it equals
\[
\sum_{j=0}^{n-1}\int\int_{\{0<t_{1}<...<t_{n}<t\}}\frac{1}{t}p_{t_{2}-t_{1}%
}(x_{1+j},x_{2+j})\ldots p_{t_{1}+t-t_{n}}(x_{n+j},x_{1+j})\prod dt_{i}dt.
\]

Performing the change of variables $v_{2}=t_{2}-t_{1},..,v_{n}=t_{n}%
-t_{n-1},v_{1}=t_{1}+t-t_{n}$, and $v=t_{1}$, we obtain:
\begin{align*}
\sum_{j=0}^{n-1}\int_{\{0<v<v_{1},0<v_{i}\}}  &  \frac{1}{v_{1}+...+v_{n}%
}p_{v_{2}}(x_{1+j},x_{2+j})\ldots p_{v_{1}}(x_{n+j},x_{1+j})\prod dv_{i}dv\\
&  =\sum_{j=0}^{n-1}\int_{\{0<v_{i}\}}\frac{v_{1}}{v_{1}+...+v_{n}}p_{v_{2}%
}(x_{1+j},x_{2+j}).......p_{v_{1}}(x_{n+j},x_{1+j})\prod dv_{i}\\
&  =\sum_{j=1}^{n}\int_{\{0<v_{i}\}}\frac{v_{j}}{v_{1}+...+v_{n}}p_{v_{2}%
}(x_{1},x_{2})\ldots p_{v_{1}}(x_{n},x_{1})\prod dv_{i}\\
&  =\int_{\{0<v_{i}\}}p_{v_{2}}(x_{1},x_{2})\ldots p_{v_{1}}(x_{n},x_{1})\prod
dv_{i}\\
&  =G^{x_{1},x_{2}}G^{x_{2},x_{3}}...G^{x_{n},x_{1}}.
\end{align*}
Note that another proof can be derived from formula (\ref{tamul})
\end{proof}

\bigskip

Let us come back to the occupation field to compute its Laplace transform.
From the Feynman-Kac formula, it comes easily that, denoting $M_{\frac{\chi
}{\lambda}}$ the diagonal matrix with coefficients $\frac{\chi_{x}}%
{\lambda_{x}}$
\[
\mathbb{P}_{t}^{x,x}(e^{-\left\langle \widehat{l},\chi\right\rangle
}-1)=\frac{1}{\lambda_{x}}(\exp(t(P-I-M_{_{\frac{\chi}{\lambda}}}))_{x}%
^{x}-\exp(t(P-I))_{x}^{x}).
\]
Integrating in $t$ after expanding, we get from the definition of $\mu$ (first
for $\chi$ small enough):%
\begin{align*}
\int(e^{-\left\langle \widehat{l},\chi\right\rangle }-1)d\mu(l)  &
=\sum_{k=1}^{\infty}\int_{0}^{\infty}[Tr((P-M_{_{\frac{\chi}{\lambda}}}%
)^{k})-Tr((P)^{k})]\frac{t^{k-1}}{k!}e^{-t}dt\\
&  \sum_{k=1}^{\infty}\frac{1}{k}[Tr((P-M_{_{\frac{\chi}{\lambda}}}%
)^{k})-Tr((P)^{k})]\\
&  =-Tr(\log(I-P+M_{_{\frac{\chi}{\lambda}}}))+Tr(\log(I-P))
\end{align*}
Hence, as $Tr(\log)=\log(\det)$%
\[
\int(e^{-\left\langle \widehat{l},\chi\right\rangle }-1)d\mu(l)=\log
[\det(-L(-L+M_{\chi/\lambda})^{-1})]=-\log\det(I+VM_{\frac{\chi}{\lambda}})
\]
which now holds for all non negative $\chi$ as both members are analytic in
$\chi$. Besides, by the ''resolvent'' equation:%
\begin{equation}
\det(I+GM_{\chi})^{-1}=\det(I-G_{\chi}M_{\chi})=\frac{\det(G_{\chi})}{\det(G)}
\label{F1}%
\end{equation}
Note that $\det(I+GM_{\chi})=\det(I+M_{\sqrt{\chi}}GM_{\sqrt{\chi}})$ and
$\det(I-G_{\chi}M_{\chi})=\det(I-M_{\sqrt{\chi}}G_{\chi}M_{\sqrt{\chi}})$, so
we can deal with symmetric matrices. Finally we have the

\begin{proposition}
\label{LAPLT}$\mu(e^{-\left\langle \widehat{l},\chi\right\rangle }%
-1)=-\log(\det(I+M_{\sqrt{\chi}}GM_{\sqrt{\chi}}))=\log(\frac{\det(G_{\chi}%
)}{\det(G)})$
\end{proposition}

\bigskip Note that in particular \bigskip$\mu(e^{-t\widehat{l}^{x}}%
-1)=-\log(1+tG^{x,x})$.\newline Note finally that if $\chi$ has support in
$D$, by the restriction property%
\[
\mu(1_{\{\widehat{l(}X\backslash D)=0\}}(e^{-<\widehat{l},\chi>}%
-1))=-\log(\det(I+M_{\sqrt{\chi}}G^{D}M_{\sqrt{\chi}}))=\log(\frac{\det
(G_{\chi}^{D})}{\det(G^{D})})
\]
Here the determinants are taken on matrices indexed by $D$ and $G^{D}$ the
Green function of the process killed on leaving $D$.

For paths we have $\mathbb{P}_{t}^{x,y}(e^{-\left\langle \widehat{l}%
,\chi\right\rangle })=\frac{1}{\lambda_{y}}\exp(t(L-M_{_{\frac{\chi}{\lambda}%
}}))_{x,y}$. Hence
\[
\mu^{x,y}(e^{-\left\langle \widehat{\gamma},\chi\right\rangle })=\frac{1}%
{\lambda_{y}}((I-P+M_{\chi/m})^{-1})_{x,y}=[G_{\chi}]^{x,y}.
\]
Also $\mathbb{E}^{x}(e^{-\left\langle \widehat{\gamma},\chi\right\rangle
})=\sum_{y}[G_{\chi}]^{x,y}\kappa_{y}$ i.e.$[G_{\chi}\kappa]^{x}$.

\section{Poisson process of loops}

\subsection{Definition}

Still following the idea of \cite{LW}, which was already implicitly in germ in
\cite{Symanz}, define, for all positive $\alpha$,\ the Poissonian ensemble of
loops $\mathcal{L}_{\alpha}$ with intensity $\alpha\mu$. We denote by
$\mathbb{P}$\ or $\mathbb{P}_{\mathcal{L}_{\alpha}}$ its distribution.

Recall it means that for any functional $\Phi$ on the loop space, vanishing on
loops of arbitrary small length,
\[
E(e^{i\sum_{l\in\mathcal{L}_{\alpha}}\Phi(l)}=\exp(\alpha\int(e^{i\Phi
(l)}-1)\mu(dl))
\]

Note that by the restriction property, $\mathcal{L}_{\alpha}^{D}%
=\{l\in\mathcal{L}_{\alpha},l\subseteq D\}$\ is a Poisson process of loops
with intensity $\mu^{D}$, and that $\mathcal{L}_{\alpha}^{D}$ \ is independent
of $\mathcal{L}_{\alpha}\backslash\mathcal{L}_{\alpha}^{D}$.

\bigskip

We denote by $\mathcal{DL}_{\alpha}$ the set of non trivial discrete loops in
$\mathcal{L}_{\alpha}$. Then,
\[
\mathbb{P(}\mathcal{DL}_{\alpha}=\{l_{1},l_{2},...l_{k}\})=e^{-\alpha\mu
(p>0)}\alpha^{k}\mu(l_{1})...\mu(l_{k})=\alpha^{k}[\frac{\det(G)}{\prod
_{x}\lambda_{x}}]^{\alpha}\prod_{x,y}C_{x,y}^{N_{x,y}^{(\alpha)}}\prod
_{x}\lambda_{x}^{-N_{x}^{(\alpha)}}%
\]
with $N_{x}^{(\alpha)}=\sum_{l\in\mathcal{L}_{\alpha}}N_{x}(l)$ and
$N_{x,y}^{(\alpha)}=\sum_{l\in\mathcal{L}_{\alpha}}N_{x,y}(l)$, when these
loops are distinct.

\bigskip

We can associate to $\mathcal{L}_{\alpha}$ a $\sigma$-finite measure (in fact
as we will see, finite when $X$ is finite, and more generally if $G$ is trace
class) called local time or occupation field
\[
\widehat{\mathcal{L}_{\alpha}}=\sum_{l\in\mathcal{L}_{\alpha}}\widehat{l}%
\]

Then, for any non-negative measure $\chi$ on $X$
\[
\mathbb{E}(e^{-\left\langle \widehat{\mathcal{L}_{\alpha}},\chi\right\rangle
})=\exp(\alpha\int(e^{-\left\langle \widehat{l},\chi\right\rangle }%
-1)d\mu(l))
\]
and therefore by proposition \ref{LAPLT} we have

\begin{corollary}
\label{lapl} $\mathbb{E}(e^{-\left\langle \widehat{\mathcal{L}_{\alpha}}%
,\chi\right\rangle })=\det(I+M_{\sqrt{\chi}}GM_{\sqrt{\chi}})^{-\alpha
}=(\frac{\det(G_{\chi})}{\det(G)})^{\alpha}$
\end{corollary}

\bigskip

Many calculations follow from this result.

It follows that $\mathbb{E}(e^{-t\widehat{\mathcal{L}_{\alpha}}^{x}%
})=(1+tG^{x,x})^{-\alpha}$. We see that $\widehat{\mathcal{L}_{\alpha}}^{x}$
follows a gamma distribution $\Gamma(\alpha,G^{x,x})$, with density
$1_{\{x>0\}}\frac{e^{-\frac{x}{G^{xx}}}}{\Gamma(\alpha)}\frac{x^{\alpha-1}%
}{(G^{xx})^{\alpha}}$ (in particular, an exponential distribution of mean
$G^{x,x}$\ for $\alpha=1$). When we let $\alpha$ vary as a time parameter, we
get a family of gamma subordinators, which can be called a ''multivariate
gamma subordinator''.

We check in particular that $\mathbb{E}(\widehat{\mathcal{L}_{\alpha}}%
^{x})=\alpha G^{x,x}$ which follows directly from $\mu(\widehat{l}%
_{x})=G^{x,x}$.

Note also that for $\alpha>1$,
\[
{\mathbb{E}((1-\exp(-\frac{\widehat{\mathcal{L}_{\alpha}}^{x}}{G^{x,x}}%
))^{-1})=\zeta(\alpha)}.
\]

More generally, for two points:
\[
\mathbb{E}(e^{-t\widehat{\mathcal{L}_{\alpha}}^{x}}e^{-s\widehat
{\mathcal{L}_{\alpha}}^{y}})=((1+tG^{x,x})(1+sG^{y,y})-st(G^{x,y}%
)^{2})^{-\alpha}%
\]

This allows to compute the joint density of $\widehat{\mathcal{L}_{\alpha}%
}^{x}$ and $\widehat{\mathcal{L}_{\alpha}}^{y}$ in terms of Bessel and Struve functions.

\bigskip

We can condition the loops by the set of associated non trivial discrete loop
by using the restricted $\sigma$-field $\sigma(\mathcal{DL}_{\alpha})$\ which
contains the variables $N_{x,y}$.\ We see from \ref{triv} and \ref{d}\ that%

\[
\mathbb{E}(e^{-\left\langle \widehat{\mathcal{L}_{\alpha}},\chi\right\rangle
}|\mathcal{DL}_{\alpha})=\prod_{x}(\frac{\lambda_{x}}{\lambda_{x}+\chi_{x}%
})^{N_{x}^{(\alpha)}+1}%
\]
The distribution of $\{N_{x}^{(\alpha)},x\in X\}$ follows easily, from
corollary \ref{lapl}\ in terms of generating functions:
\begin{equation}
\mathbb{E}(\prod_{x}(s_{x}^{N_{x}^{(\alpha)}+1})=\det(\delta_{x,y}%
+\sqrt{\frac{\lambda_{x}\lambda_{y}(1-s_{x})(1-s_{y})}{s_{x}s_{y}}}%
G_{x,y})^{-\alpha} \label{GF}%
\end{equation}
so that the vector of components $N_{x}^{(\alpha)}$ follows a multivariate
negative binomial distribution (see for example \cite{VJ2}).

It follows in particular that $N_{x}^{(\alpha)}$\ follows a negative binomial
distribution of parameters $-\alpha$ and $\frac{1}{\lambda_{x}G^{xx}}$. Note
that for $\alpha=1$, $N_{x}^{(1)}+1$ follows a geometric distribution of
parameter $\frac{1}{\lambda_{x}G^{xx}}$.

\subsection{Moments and polynomials of the occupation field}

It is easy to check (and well known from the properties of the gamma
distributions) that the moments of $\widehat{\mathcal{L}_{\alpha}}^{x}$ are
related to the factorial moments of $N_{x}^{(\alpha)}$\ :%
\[
\mathbb{E}((\widehat{\mathcal{L}_{\alpha}}^{x})^{k}|\mathcal{DL}_{\alpha
})=\frac{(N_{x}^{(\alpha)}+k)(N_{x}^{(\alpha)}+k-1)...(N_{x}^{(\alpha)}%
+1)}{k!\lambda_{x}^{k}}%
\]

\bigskip

It is well known that Laguerre polynomials $L_{k}^{(\alpha-1)}$ with
generating function
\[
\sum_{0}^{\infty}t^{k}L_{k}^{(\alpha-1)}(u)=\frac{e^{-\frac{ut}{1-t}}%
}{(1-t)^{\alpha}}%
\]
are orthogonal for the $\Gamma(\alpha,1)$ distribution with density
$\frac{u^{\alpha-1}e^{-u}}{\Gamma(\alpha)}1_{\{u>0\}}$. They have mean zero
and variance $\frac{\Gamma(\alpha+k)}{k!}$. Hence if we set $\sigma
_{x}=G^{x,x}$and $P_{k}^{\alpha,\sigma}(x)=(-\sigma)^{k}L_{k}^{(\alpha
-1)}(\frac{x}{\sigma})$, the random variables $P_{k}^{\alpha,\sigma_{x}%
}(\widehat{\mathcal{L}_{\alpha}}^{x})$ are orthogonal with mean $0$\ and
variance $\sigma^{2k}\frac{\Gamma(\alpha+k)}{k!}$, for $k>0$.

Note that $P_{1}^{\alpha,\sigma_{x}}(\widehat{\mathcal{L}_{\alpha}}%
^{x})=\widehat{\mathcal{L}_{\alpha}}^{x}-\alpha\sigma_{x}=\widehat
{\mathcal{L}_{\alpha}}^{x}-\mathbb{E}(\widehat{\mathcal{L}_{\alpha}}^{x})$.
\emph{It will be denoted }$\widetilde{\mathcal{L}_{\alpha}}^{x}$\emph{.}

Moreover, we have $\sum_{0}^{\infty}t^{k}P_{k}^{\alpha,\sigma}(u)=\sum(-\sigma
t)^{k}L_{k}^{(\alpha-1)}(\frac{u}{\sigma})=\frac{e^{\frac{ut}{1+\sigma t}}%
}{(1+\sigma t)^{\alpha}}$

Note that
\begin{multline*}
\mathbb{E}(\frac{e^{\frac{\widehat{\mathcal{L}_{\alpha}}^{x}t}{1+\sigma_{x}t}%
}}{(1+\sigma_{x}t)^{\alpha}}\frac{e^{\frac{\widehat{\mathcal{L}_{\alpha}}%
^{y}s}{1+\sigma_{y}s}}}{(1+\sigma_{y}s)^{\alpha}})\\
=\frac{1}{(1+\sigma_{x}t)^{\alpha}(1+\sigma_{y}s)^{\alpha}}((1-\frac{\sigma
_{x}t}{1+\sigma_{x}t})(1-\frac{\sigma_{y}s}{1+\sigma_{y}s})-\frac{t}%
{1+\sigma_{x}t}\frac{s}{1+\sigma_{y}s}((G^{x,y})^{2})^{-\alpha}\\
=(1-st(G^{x,y})^{2})^{-\alpha}.
\end{multline*}
Therefore, we get, by developping in entire series in $(s,t)$ and identifying
the coefficients:%

\begin{equation}
\mathbb{E}(P_{k}^{\alpha,\sigma_{x}}(\widehat{\mathcal{L}_{\alpha}}^{x}%
),P_{l}^{\alpha,\sigma_{y}}(\widehat{\mathcal{L}_{\alpha}}^{y}))=\delta
_{k,l}(G^{x,y})^{2k}\frac{\alpha(\alpha+1)...(\alpha+k-1)}{k!} \label{corr}%
\end{equation}

Let us stress the fact that $G^{x,x}$ and $G^{y,y}$ do not appear on the right
side of this this formula. This is quite important from the renormalisation
point of view, as we will consider in the last section the two dimensional
Brownian motion for which the Green function diverges on the diagonal.

More generally one can prove similar formulas for products of higher order.

Note that since $G_{\chi}M_{\chi}$\ is a contraction, from determinant
expansions given in \cite{VJ1} and \cite{VJ2}, we have%
\[
\det(I+M_{\sqrt{\chi}}GM_{\sqrt{\chi}})^{-\alpha}=1+\sum_{k=1}^{\infty
}(-1)^{k}\sum\chi_{i_{1}}...\chi_{i_{k}}Per_{\alpha}(G_{i_{l},i_{m}},1\leq
l,m\leq k)
\]
and then, from corollary \ref{lapl}, it comes that:%
\[
\mathbb{E}(\left\langle \widehat{\mathcal{L}_{\alpha}},\chi\right\rangle
^{k})=\sum\chi_{i_{1}}...\chi_{i_{k}}Per_{\alpha}(G_{i_{l},i_{m}},1\leq
l,m\leq k)
\]
Here the $\alpha$-permanent $Per_{a}$ is defined as $\sum_{\sigma
\in\mathcal{S}_{k}}\alpha^{m(\sigma)}G_{i_{1},i_{\sigma(1)}}...G_{i_{k}%
,i_{\sigma(k)}}$ with $m(\sigma)$ denoting the number of cycles in $\sigma$.

Note that from this determinant expansion follows directly (see \cite{VJ2}) an
explicit form for the multivariate negative binomial distribution, and
therefore, a series expansion for the density of the multivariate gamma distribution.

\bigskip

It is actually not difficult to give a direct proof of this result. Thus, the
Poisson process of loops provides a natural probabilistic proof and
interpretation of this combinatorial identity (see \cite{VJ2} for an
historical view of the subject).

\bigskip We can show in fact that:

\begin{proposition}
For any $(i_{1},...i_{k})$ in $X^{k}$, $\mathbb{E}(\widehat{\mathcal{L}%
_{\alpha}}^{i_{1}}...\widehat{\mathcal{L}_{\alpha}}^{i_{k}})=Per_{\alpha
}(G^{i_{l},i_{m}},1\leq l,m\leq k)$
\end{proposition}

\begin{proof}
The cycles of the permutations in the expression of $Per_{\alpha}$\ are
associated with point configurations on loops. We obtain the result by summing
the contributions of all possible partitions of the points $i_{1}...i_{k}$
into a finite set of distinct loops.\ We can then decompose again the
expression according to ordering of points on each loop. We can conclude by
using the formula $\mu(\widehat{l}^{x_{1},...,x_{m}})=G^{x_{1},x_{2}}%
G^{x_{2},x_{3}}...G^{x_{m},x_{1}}$ and the following property of Poisson
measures (Cf formula 3-13 in \cite{King}): For any system of non negative loop
functionals $F_{i}$%

\[
\mathbb{E}(\sum_{l_{1}\neq l_{2}...\neq l_{k}\in\mathcal{L}_{\alpha}}\prod
F_{i}(l_{i}))=\prod\alpha\mu(F_{i})
\]
\end{proof}

\bigskip

\begin{remark}
We can actually check this formula in the special case $i_{1}=i_{2}%
=...=i_{k}=x$. From the moments of the Gamma distribution, we have that
$\mathbb{E}((\widehat{\mathcal{L}_{\alpha}}^{x})^{n})=(G^{x,x})^{n}%
\alpha(\alpha+1)...(\alpha+n-1)$ and the $\alpha$-permanent writes $\sum
_{1}^{n}d(n,k)\alpha^{k}$ where the coefficients $d(n,k)$ are the numbers of
$n-$permutations with $k$ cycles (Stirling numbers of the first kind). One
checks that $d(n+1,k)=nd(n,k)+d(n,k-1)$.
\end{remark}

\bigskip

Let $\mathcal{S}_{k}^{0}$ be the set of permutations of $k$\ elements without
fixed point. They correspond to configurations without isolated point.

Set $Per_{\alpha}^{0}(G^{i_{l},i_{m}},1\leq l,m\leq k)=\sum_{\sigma
\in\mathcal{S}_{k}^{0}}\alpha^{m(\sigma)}G^{i_{1},i_{\sigma(1)}}%
...G^{i_{k},i_{\sigma(k)}}$. Then an easy calculation shows that:

\begin{corollary}
\label{mom}$\mathbb{E}(\widetilde{\mathcal{L}_{\alpha}}^{i_{1}}...\widetilde
{\mathcal{L}_{\alpha}}^{i_{k}})=Per_{\alpha}^{0}(G^{i_{l},i_{m}},1\leq l,m\leq k)\!$
\end{corollary}

\begin{proof}
Indeed, the expectation writes
\[
\sum_{p\leq k}\sum_{I\subseteq\{1,...k\},\left|  I\right|  =p}(-1)^{k-p}%
\prod_{l\in I^{c}}G^{i_{l},i_{l}}Per_{\alpha}(G^{i_{a},i_{b}},a,b\in I)
\]
and
\[
Per_{\alpha}(G^{i_{a},i_{b}},a,b\in I) =\sum_{J\subseteq I}\prod_{j\in
I\backslash J}G^{j,j}Per_{\alpha}^{0}(G^{i_{a},i_{b}},a,b\in J).
\]
Then, expressing $\mathbb{E}(\widetilde{\mathcal{L}_{\alpha}}^{i_{1}%
}...\widetilde{\mathcal{L}_{\alpha}}^{i_{k}})$ in terms of $Per_{\alpha}^{0}%
$'s, we see that if $J\subseteq\{1,...k\}$, ${\left|  J\right|  <k}$, the
coefficient of $Per_{\alpha}^{0}(G^{i_{a},i_{b}},a,b\in J)$ is $\displaystyle
{\sum_{I,I\supseteq J}(-1)^{k-\left|  I\right|  }\prod_{j\in J^{c}}%
G^{i_{j},i_{j}}}$ which vanishes as $(-1)^{-\left|  I\right|  }=(-1)^{\left|
I\right|  }=(-1)^{\left|  J\right|  }(-1)^{\left|  I\backslash J\right|  }$
and $\sum_{I\supseteq J}(-1)^{\left|  I\backslash J\right|  }=(1-1)^{k-\left|
J\right|  }=0$.
\end{proof}

\bigskip

Set $Q_{k}^{\alpha,\sigma}(u)=P_{k}^{\alpha,\sigma}(u+\alpha\sigma)$ so that
$P_{k}^{\alpha,\sigma}(\widehat{\mathcal{L}_{\alpha}}^{x})=Q_{k}%
^{\alpha,\sigma}(\widetilde{\mathcal{L}_{\alpha}}^{x})$. This quantity will be
called the $n$-th renormalized self intersection local time or the $n$-th
renormalized power of the occupation field and denoted $\widetilde
{\mathcal{L}}_{\alpha}^{x,n}$.

From the recurrence relation of Laguerre polynomials
\[
nL_{n}^{(\alpha-1)}(u)=(-u+2n+\alpha-2)L_{n-1}^{(\alpha-1)}-(n+\alpha
-2)L_{n-2}^{(\alpha-1)},
\]
we get that
\[
nQ_{n}^{\alpha,\sigma}(u)=(u-2\sigma(n-1))Q_{n-1}^{\alpha,\sigma}%
(u)-\sigma^{2}(\alpha+n-2)Q_{n-2}^{\alpha,\sigma}(u)
\]
In particular $Q_{2}^{\alpha,\sigma}(u)=\frac{1}{2}(u^{2}-2\sigma
u-\alpha\sigma^{2})$.\newline We have also, from (\ref{corr})%

\begin{equation}
\mathbb{E}(Q_{k}^{\alpha,\sigma_{x}}(\widetilde{\mathcal{L}_{\alpha}}%
^{x}),Q_{l}^{\alpha,\sigma_{y}}(\widetilde{\mathcal{L}_{\alpha}}^{y}%
))=\delta_{k,l}(G^{x,y})^{2k}\frac{\alpha(\alpha+1)...(\alpha+k-1)}{k!}\!
\label{orth}%
\end{equation}

The comparison of the identity (\ref{orth}) and corollary \ref{mom} \ yields a
combinatorial result which will be fundamental in the renormalizing procedure
presented in the last section.

The identity (\ref{orth}) can be considered as a polynomial identity in the
variables $\sigma_{x}$, $\sigma_{y}$ and $G^{x,y}$.

If $Q_{k}^{\alpha,\sigma_{x}}(u)=\sum_{m=0}^{k}q_{m}^{\alpha,k}u^{m}\sigma
_{x}^{k-m}$, if we denote $N_{n,m,r,p}$ the number of ordered configurations
of $n$ black points and $m$ red points on $r$\ non trivial oriented cycles,
such that only $2p$ links are between red and black points, we have%

\[
\mathbb{E}((\widetilde{\mathcal{L}_{\alpha}}^{x})^{n}(\widetilde
{\mathcal{L}_{\alpha}}^{y})^{m})=\sum_{r}\sum_{p\leq\inf(m,n)}\alpha
^{r}N_{n,m,r,p}(G^{x,y})^{2p}(\sigma_{x})^{n-p}(\sigma_{y})^{m-p}%
\]
and therefore%
\begin{align}
\sum_{r}\sum_{p\leq m\leq k}\sum_{p\leq n\leq l}\alpha^{r}q_{m}^{\alpha
,k}q_{n}^{\alpha,l}N_{n,m,r,p}  &  =0\text{ \negthinspace unless
}p=l=k.\label{null}\\
\sum_{r}\alpha^{r}q_{k}^{\alpha,k}q_{k}^{\alpha,k}N_{k,k,r,k}  &
=\frac{\alpha(\alpha+1)...(\alpha+k-1)}{k!} \label{factup}%
\end{align}
Note that one can check directly that $q_{k}^{\alpha,k}=\frac{1}{k!}$, and
$N_{k,k,1,k}=k!(k-1)!$, $N_{k,k,k,k}=k!$ which confirms the identity
(\ref{factup}) above.

\subsection{\label{hit}Hitting probabilities}

Let $[H^{F}]_{y}^{x}=\mathbb{P}_{x}(x_{T_{F}}=y)$ be the hitting distribution
of $F$ by the Markov chain starting at $x$. Set $D=F^{c}$ and denote $e^{D}$,
$P^{D}=P)|_{D\times D}$, $V^{D}=[(I-P^{D})]^{-1}$ and $G^{D}=[(M_{\lambda
}-C)|_{D\times D}]^{-1}$ the energy, the transtion matrix, the potential and
the Green function of the process killed at the hitting of $F$. Recall that

$[H^{F}]_{y}^{x}$ $=1_{\{x=y\}}+\sum_{0}^{\infty}\sum_{z\in D}[(P^{D}%
)^{k}]_{z}^{x}P_{y}^{z}=1_{\{x=y\}}+\sum_{0}^{\infty}\sum_{z\in D}[V^{D}%
]_{z}^{x}P_{y}^{z}$. Moreover we have by the strong Markov property,
$V=V^{D}+H^{F}V$ and therefore $G=$ $G^{D}+H^{F}G$. (Here we extend $V^{D}$
and $G^{D}$ to $X\times X$ by adding zero entries outside $D\times D$).

As $G$\ and $G^{D}$ are symmetric, we have $[H^{F}G]_{y}^{x}=[H^{F}G]_{x}^{y}$
so that for any measure $\nu$, $H^{F}(G\nu)=G(\nu H^{F})$.

Therefore we see that for any function $f$ and measure $\nu$, $e(H^{F}%
f,G^{D}\nu)=e(H^{F}f,G\nu)-e(H^{F}f,H^{F}G\nu)=\left\langle H^{F}%
f,\nu\right\rangle -e(H^{F}f,G(H^{F}\nu))=0$ as $(H^{F})^{2}=H^{F}$.

Equivalently, we have the following:

\begin{proposition}
\label{proj}For any $g$ vanishing on $F$, $e(H^{F}f,g)=0$ so that $I-H^{F}$ is
the $e$-orthogonal projection on the space of functions supported in $D$.
\end{proposition}

For further developments see for example ( \cite{LJ0}) and its references.

The restriction property holds for $\mathcal{L}_{\alpha}$ as it holds for
$\mu$. The set $\mathcal{L}_{\alpha}^{D}$ of loops inside $D$ is associated
with $\mu^{D}$ and independent of $\mathcal{L}_{\alpha}-\mathcal{L}_{\alpha
}^{D}$. Therefore, we see from corollary \ref{lapl} that
\[
\mathbb{E}(e^{-\left\langle \widehat{\mathcal{L}_{\alpha}}-\widehat
{\mathcal{L}_{\alpha}^{D}},\chi\right\rangle })=(\frac{\det(G_{\chi})}%
{\det(G)}\frac{\det(G^{D})}{\det(G_{\chi}^{D})})^{\alpha}.
\]

From the support of the of the Gamma distribution, we see that $\mu
(\widehat{l(}F)>0)=\infty$. But this is clearly due to trivial loops as it can
be seen directly from the definition of $\mu$ that in this simple framework
they cover the whole space $X$.

Note however that
\begin{align*}
\mu(\widehat{l}(F)>0,p>1)  &  =\mu(p>1)-\mu(\widehat{l(}F)=0,p>1)
=\mu(p>1)-\mu^{D}(p>1)\\
&  =-\log(\frac{\det(I-P)}{\det_{D\times D}(I-P)})=-\log(\frac{\det(G^{D}%
)}{\prod_{x\in F}\lambda_{x}\det(G)}).
\end{align*}
It follows that the probability no non trivial loop (i.e.a loop which is not
reduced to a point) in $\mathcal{L}_{\alpha}$ intersects $F$ equals
\[
\exp(-\alpha\mu(\{l,\!p(l)>1,\widehat{l}(F)>0\}))=(\frac{\det(G^{D})}%
{\prod_{x\in F}\lambda_{x}\det(G)})^{\alpha}.
\]

Recall that by Jacobi's identity, for any $(n+p,n+p)$ invertible matrix $A$,
\[
\det(A^{-1})\det(A_{ij},1\leq i,j\leq n)=\det((A^{-1})_{k,l},n\leq k,l\leq
n+p).
\]
In particular, $\displaystyle  {\det(G^{D})=\frac{\det(G)}{\det(G|_{F\times
F})}}$, so we have the

\begin{proposition}
The probability that no non trivial loop in $\mathcal{L}_{\alpha}$ intersects
$F$ equals
\[
[\prod_{x\in F}\lambda_{x}\det_{F\times F}(G)]^{-\alpha}.
\]
Moreover $\displaystyle                         {\mathbb{E}(e^{-\left\langle
\widehat{\mathcal{L}_{\alpha}}-\widehat{\mathcal{L}_{\alpha}^{D}}%
,\chi\right\rangle })=(\frac{\det_{F\times F}(G_{\chi})}{\det_{F\times F}%
(G)})^{\alpha}}$
\end{proposition}

\bigskip

In particular, it follows that the probability no non trivial loop in
$\mathcal{L}_{\alpha}$ visits $x$ equals $(\frac{1}{\lambda_{x}G^{x,x}%
})^{\alpha}$ which is also aconsequence of the fact that $N_{x}$ follows a
negative binomial distribution of parameters $-\alpha$ and $\frac{1}%
{\lambda_{x}G^{x,x}}$

Also, if $F_{1}$ and $F_{2}$ are disjoint,
\begin{align*}
\mu(\widehat{l(}F_{1})\widehat{l(}F_{2})>0)  &  =\mu(\widehat{l(}%
F_{1})>0,p>1)+\mu(\widehat{l(}F_{2})>0,p>1)-\mu(\widehat{l(}F_{1}\cup
F_{2})>0,p>1)\\
&  =\log(\frac{\det(G)\det(G^{D_{1}\cap D_{2}})}{\det(G^{D_{1}})\det(G^{D_{2}%
})}).
\end{align*}
Therefore the probability no non trivial loop in $\mathcal{L}_{\alpha}$
intersects $F_{1}$ and $F_{2}$ equals
\[
\exp(-\alpha\mu(\{l,\!p(l)>1,\prod\widehat{l(}F_{i})>0\}))=(\frac{\det
(G)\det(G^{D_{1}\cap D_{2}})}{\det(G^{D_{1}})\det(G^{D_{2}})})^{-\alpha}%
\]
It follows that the probability no non trivial loop in $\mathcal{L}_{\alpha}$
visits two distinct points $x$ and $y$ equals $(\frac{G^{x,x}G^{y,y}%
-(G^{x,y})^{2}}{G^{x,x}G^{y,y}})^{\alpha}$ and in particular $1-\frac{(G^{x,y}%
)^{2}}{G^{x,x}G^{y,y}}$ if $\alpha=1$. This formula can be easily generalized
to $n$ disjoint sets.

\section{The Gaussian free field}

\subsection{Dynkin's Isomorphism\bigskip}

By a well known calculation, if $X$ is finite, for any $\chi\in\mathbb{R}%
_{+}^{X}$,
\[
\frac{\sqrt{\det(M_{\lambda}-C)}}{(2\pi)^{\left|  X\right|  /2}}%
\int(e^{-\frac{1}{2}<z,\chi>}e^{-\frac{1}{2}e(z)}\Pi_{u\in X}dz^{u}%
=\sqrt{\frac{\det(G_{\chi})}{\det(G)}}%
\]
and%
\[
\frac{\sqrt{\det(M_{\lambda}-C)}}{(2\pi)^{\left|  X\right|  /2}}\int
z^{x}z^{y}(e^{-\frac{1}{2}<z^{2},\chi>}e^{-\frac{1}{2}e(z)}\Pi_{u\in X}%
dz^{u}=(G_{\chi})^{x,y}\sqrt{\frac{\det(G_{\chi})}{\det(G)}}%
\]

This can be easily reformulated by introducing the Gaussian field $\phi$
defined by the covariance $\mathbb{E}_{\phi}\mathbb{(}\phi^{x}\phi
^{y})=G^{x,y}$ (this reformulation cannot be dispensed with when $X$ becomes infinite)

So we have $\mathbb{E(}(e^{-\frac{1}{2}<\phi^{2},\chi>})=\det(I+GM_{_{\chi}%
})^{-\frac{1}{2}}=\sqrt{\det(G_{\chi}G^{-1})}$ and

$\mathbb{E(}(\phi^{x}\phi^{y}e^{-\frac{1}{2}<\phi^{2},\chi>})=(G_{\chi}%
)^{x,y}\sqrt{\det(G_{\chi}G^{-1})}$ Then as sums of exponentials of the form
$e^{-\frac{1}{2}<\cdot,\chi>}$\ \ are dense in continuous functions on
$\mathbb{R}_{+}^{X}$\ the following holds:

\begin{theorem}
\label{iso}

\begin{enumerate}
\item[a)] The fields $\widehat{\mathcal{L}_{\frac{1}{2}}}$ and $\frac{1}%
{2}\phi^{2}$ have the same distribution.

\item[b)] $\mathbb{E}_{\phi}\mathbb{(}(\phi^{x}\phi^{y}F(\frac{1}{2}\phi
^{2}))=\int\mathbb{E}(F(\widehat{\mathcal{L}_{1}}+\widehat{\gamma}))\mu
^{x,y}(d\gamma)$ for any bounded functional $F$ of a non negative field.
\end{enumerate}
\end{theorem}

\bigskip

\noindent\textbf{Remarks:}

a) This is a version of Dynkin's isomorphism (Cf \cite{Dy}). It can be
extended to non symmetric generators (Cf \cite{LJ2}).

b) An analogous result can be given when $\alpha$ is any positive half
integer, by using real vector valued Gaussian field, or equivalently complex
fields for integral values of $\alpha$ (in particular $\alpha=1)$.

c) Note it implies immediately that the process $\phi^{2}$ is infinitely
divisible. See \cite{EK} and its references for a converse and earlier proofs
of this last fact.

\bigskip

\subsection{Fock spaces and Wick product}

The Gaussian space $\mathcal{H}$ spanned by $\{\phi^{x},x\in X\}$ is
isomorphic to the Dirichlet space $\mathbb{H}$ by the linear map mapping
$\phi^{x}$ on $G^{x,.}$ which extends into an isomorphism between \ the space
of square integrable functionals of the Gaussian fields and the symmetric Fock
space obtained as the closure of the sum of all symmetric tensor powers of
$\mathbb{H}$ (Bose second quantization: See \cite{Sim2}, \cite{Nev}). We have
seen in theorem \ref{iso}\ that $L^{2}$ functionals of $\widehat
{\mathcal{L}_{1}}$ can be represented in this symmetric Fock space.

In order to prepare the extension of these isomorphisms to the more difficult
framework of continuous spaces (which can often be viewed as scaling limits of
discrete spaces), including especially the planar Brownian motion considered
in \cite{LW}, we shall introduce the renormalized (or Wick) powers of $\phi$.
We set $:(\phi^{x})^{n}:=(G^{x,x})^{\frac{n}{2}}H_{n}(\phi^{x}/\sqrt{G^{x,x}%
})$ where $H_{n}$ in the $n$-th Hermite polynomial (characterized by
$\sum\frac{t^{n}}{n!}H_{n}(u)=e^{tu-\frac{t^{2}}{2}}$). It is the inverse
image of the $n$-th tensor\ power of $G^{x,.}$\ in the Fock space.

Setting as before $\sigma_{x}=G^{x,x}$, from the relation between Hermite
polynomials $H_{2n}$\ and Laguerre polynomials $L_{n}^{-\frac{1}{2}}$,%

\[
H_{2n}(x)=(-2)^{n}n!L_{n}^{-\frac{1}{2}}(\frac{x^{2}}{2})
\]
it comes that:%
\[
:(\phi^{x})^{2n}:=2^{n}n!P_{n}^{\frac{1}{2},\sigma}((\frac{(\phi^{x})^{2}}%
{2}))
\]

More generally, if $\phi_{1},\phi_{2}...\phi_{k}$ \ are $k$ independent copies
of the free field, we can define

$:\prod_{j=1}^{k}\phi_{j}^{n_{j}}:\;=\prod_{j=1}^{k}:\phi_{j}^{n_{j}}:$. Then
it comes that:%
\[
:(\sum_{1}^{k}\phi_{j}^{2})^{n}:=\sum_{n_{1}+..+n_{k}=n}\frac{n!}%
{n_{1}!...n_{k}!}\prod_{j=1}^{k}:\phi_{j}^{2n_{j}}:
\]

From the generating function of the polynomials $P_{n}^{\frac{k}{2},\sigma}$,
\[
P_{n}^{\frac{k}{2},\sigma}(\sum_{1}^{k}u_{j})=\sum_{n_{1}+..+n_{k}=n}%
\frac{n!}{n_{1}!...n_{k}!}\prod_{j=1}^{k}P_{n_{j}}^{\frac{1}{2},\sigma}%
(u_{j}).
\]
Therefore%
\begin{equation}
P_{n}^{\frac{k}{2},\sigma}(\frac{\sum(\phi_{j})^{2}}{2})=\frac{1}{2^{n}%
n!}:(\sum_{1}^{k}\phi_{j}^{2})^{n}: \label{polywick}%
\end{equation}
Note that $:\sum_{1}^{k}\phi_{j}^{2}:=\sum_{1}^{k}\phi_{j}^{2}-\sigma$ These
variables are orthogonal in $L^{2}$. Let $\widetilde{l}^{x}=\widehat{l}%
^{x}-\sigma$ be the centered occupation field. Note that an equivalent
formulation of theorem \ref{iso} is that the fields \quad$\frac{1}{2}:\sum
_{1}^{k}\phi_{j}^{2}:$\quad and $\widetilde{\mathcal{L}}_{\frac{k}{2}}$ have
the same law.

Let us now consider the relation of higher Wick powers with self intersection
local times.

Recall that the renormalized $n$-th self intersections field $\widetilde
{\mathcal{L}}_{1}^{x,n}=P_{n}^{\alpha,\sigma}(\widehat{\mathcal{L}_{\alpha}%
}^{x})=Q_{n}^{\alpha,\sigma}(\widetilde{\mathcal{L}_{\alpha}}^{x})$\ have been
defined by orthonormalization in $L^{2}$\ of the powers of the occupation time.

Then comes the

\begin{proposition}
The fields $\widetilde{\mathcal{L}}_{\frac{k}{2}}^{\cdot,n}$ and
$:(\frac{1}{n!2^{n}}\sum_{1}^{k}\phi_{j}^{2})^{n}:$ have the same law.
\end{proposition}

This follows directly from (\ref{polywick}).

\begin{remark}
As a consequence, it can be shown that:
\[
\mathbb{E}(\prod_{j=1}^{r}Q_{k_{j}}^{\alpha,\sigma_{x_{j}}}(\widetilde
{\mathcal{L}_{\alpha}}^{x_{j}}))=\sum_{\sigma\in\mathcal{S}_{k_{1}%
,k_{2},...k_{j}}}(2\alpha)^{m(\sigma)}G^{i_{1},i_{\sigma(1)}}...G^{i_{k}%
,i_{\sigma(k)}}%
\]
where $\mathcal{S}_{k_{1},k_{2},...k_{j}}$ is the set of permutations $\sigma$
of $k=\sum k_{j}$ such that

${\sigma(\{\sum_{1}^{j-1}k_{l}+1,...\sum_{1}^{j-1}k_{l}+k_{j}\}\cap\{\sum
_{1}^{j-1}k_{l}+1,...\sum_{1}^{j-1}k_{l}+k_{j}\}}$ is empty for all $j$.
\end{remark}

The identity follows from Wick's theorem when $\alpha$ is a half integer, then
extends to all $\alpha$ since both members are polynomials in $\alpha$. The
condition on $\sigma$ indicates that no pairing is allowed inside the same
Wick power.

\section{Energy variation and currents}

The loop measure $\mu$ depends on the energy $e$ which is defined by the free
parameters $C,\kappa$. It will sometimes be denoted $\mu_{e}$. We shall denote
$\mathcal{Z}_{e}$ the determinant $\det(G)=\det(M_{\lambda}-C)^{-1}$. Then
$\mu(p>0)=\log(\mathcal{Z}_{e})+\sum\log(\lambda_{x})$.

$\mathcal{Z}_{e}^{\alpha}$ is called the partition function of $\mathcal{L}%
_{\alpha}$.

The following result is suggested by an analogy with quantum field theory (Cf
\cite{Gaw}).

\begin{proposition}
\begin{enumerate}

\item[i)] $\frac{\partial\mu}{\partial\kappa_{x}}=\widehat{l}^{x}\mu$

\item[ii)] If $C_{x,y}>0$, $\frac{\partial\mu}{\partial C_{x,y}}=-T_{x,y}\mu$
with $T_{x,y}(l)=(\widehat{l}^{x}+\widehat{l}^{y})-\frac{N_{x,y}}{C_{x,y}%
}(l)-\frac{N_{y,x}}{C_{x,y}}(l)$.
\end{enumerate}
\end{proposition}

Note that the formula i) would be a direct consequence of the Dynkin
isomorphism if we considered only sets defined by the occupation field.

\begin{proof}
Recall that by formula (\ref{d}): $\mu^{\ast}(p=1,\xi=x,\widehat{\tau}\in
dt)=e^{-\lambda_{x}t}\frac{dt}{t}$ and $\mu^{\ast}(p=k,\xi_{i}=x_{i}%
,\widehat{\tau}_{i}\in dt_{i})=\frac{1}{k}\prod_{x,y}C_{x,y}^{N_{x,y}}%
\prod_{x}\lambda_{x}^{-N_{x}}\prod_{i\in\mathbb{Z}/p\mathbb{Z}}\lambda
_{\xi_{i}}e^{-\lambda_{\xi_{i}}t_{i}}dt_{i}$

Moreover we have $C_{x,y}=C_{y,x}=\lambda_{x}P_{y}^{x}$ and $\lambda
_{x}=\kappa_{x}+\sum_{y}C_{x,y}$

The two formulas follow by elementary calculation.
\end{proof}

 Recall that $\mu(\widehat{l}^{x})=G^{x,x}$ and $\mu(N_{x,y}%
)=G^{x,y}C_{x,y}$.\newline So we have $\mu(T_{x,y})=G^{x,x}+G^{y,y}-2G^{x,y}%
$.\newline Then, the above proposition allows to compute all moments of $T$
and $\widehat{l}$ relative to $\mu_{e}$ (they could be called Schwinger
functions). The above proposition gives the infinitesimal form of the
following formula.

\begin{proposition}
Consider another energy form $e^{\prime}$ defined on the same graph. Then we
have the following identity:%
\[
\frac{\partial\mu_{e^{\prime}}}{\partial\mu_{e}}=e^{\sum N_{x,y}%
\log(\frac{C_{x,y}^{\prime}}{C_{x,y}})-\sum(\lambda_{x}^{\prime}-\lambda
_{x})\widehat{l}^{x}}%
\]
Consequently
\begin{equation}
\mu_{e}((e^{\sum N_{x,y}\log(\frac{C_{x,y}^{\prime}}{C_{x,y}})-\sum
(\lambda_{x}^{\prime}-\lambda_{x})\widehat{l}^{x}}-1))=\log(\frac{\mathcal{Z}%
_{e^{\prime}}}{\mathcal{Z}_{e}}) \label{mupr}%
\end{equation}
\end{proposition}

\begin{proof}
The first formula is a straightforward consequence of (\ref{d}). The proof of
(\ref{mupr})\ goes by evaluating separately the contribution of trivial loops,
which equals $\sum_{x}\log(\frac{\lambda_{x}}{\lambda_{x}^{\prime}})$.
Indeed,
\begin{multline*}
\mu_{e}((e^{\sum N_{x,y}\log(\frac{C_{x,y}^{\prime}}{C_{x,y}})-\sum
(\lambda_{x}^{\prime}-\lambda_{x})\widehat{l}^{x}}-1))=\mu_{e^{\prime}%
}(p>1)-\mu_{e}(p>1)\\+\mu_{e}(1_{\{p=1\}}(e^{\sum(\lambda_{x}^{\prime}%
-\lambda_{x})\widehat{l}^{x}}-1)).
\end{multline*}

The difference of the first two terms equals $\log(\mathcal{Z}_{e^{\prime}%
})+\sum\log(\lambda_{x}^{\prime})-(\log(\mathcal{Z}_{e})-\sum\log(\lambda
_{x}))$. The last term equals $\sum_{x}\int_{0}^{\infty}(e^{-\frac{\lambda
_{x}^{\prime}-\lambda_{x}}{\lambda_{x}}t}-1)\frac{e^{-t}}{t}dt$\ which can be
computed as before:
\begin{equation}
\mu_{e}(1_{\{p=1\}}(e^{\sum(\lambda_{x}^{\prime}-\lambda_{x})\widehat{l}^{x}%
}-1))=-\sum\log(\frac{\lambda_{x}^{\prime}}{\lambda_{x}}) \label{unpt}%
\end{equation}
\end{proof}

\bigskip

\begin{remark}
(h-transforms) Note that if $C_{x,y}^{^{\prime}}=h^{x}h^{y}C_{x,y}$ and
$\kappa_{x}^{\prime}=-hLh\lambda$ for some positive function $h$ on $E$ such
that $Lh\leq 0$, as $\lambda^{\prime}=h^{2}\lambda$ and $[P^{\prime}]_{y}%
^{x}=\frac{1}{h^{x}}P_{y}^{x}h^{y}$, we have $[G^{\prime}]^{x,y}%
=\frac{G^{x,y}}{h^{x}h^{y}}$ and $\frac{\mathcal{Z}_{e^{\prime}}}%
{\mathcal{Z}_{e}}=\frac{1}{\prod(h^{x})^{2}}$.\newline
\end{remark}

\begin{remark}
Note also that $[\frac{\mathcal{Z}_{e^{\prime}}}{\mathcal{Z}_{e}}%
]^{\frac{1}{2}}=\mathbb{E(}e^{-\frac{1}{2}[e^{\prime}-e](\phi)})$, if $\phi$
is the Gaussian free field associated with $e$.\newline
\end{remark}

\bigskip

Integrating out the holding times, formula (\ref{mupr})\ can be written
equivalently:%
\begin{equation}
\mu_{e}(\prod_{(x,y)}[\frac{C_{x,y}^{\prime}}{C_{x,y}}]^{N_{x,y}}\prod
_{x}[\frac{\lambda_{x}}{\lambda_{x}^{\prime}}]^{N_{x}+1}-1)=\log
(\frac{\mathcal{Z}_{e^{\prime}}}{\mathcal{Z}_{e}}) \label{F}%
\end{equation}
and therefore%
\[
\mathbb{E}_{_{\mathcal{L}_{\alpha}}}(\prod_{(x,y)}[\frac{C_{x,y}^{\prime}%
}{C_{x,y}}]^{N_{x,y}^{(\alpha)}}\prod_{x}[\frac{\lambda_{x}}{\lambda
_{x}^{\prime}}]^{N_{x}^{(\alpha)}+1})=\mathbb{E}_{_{\mathcal{L}_{\alpha}}%
}(\prod_{(x,y)}[\frac{C_{x,y}^{\prime}}{C_{x,y}}]^{N_{x,y}^{(\alpha)}%
}e^{-\left\langle \lambda^{\prime}-\lambda,\widehat{\mathcal{L}_{\alpha}%
}\right\rangle }=(\frac{\mathcal{Z}_{e^{\prime}}}{\mathcal{Z}_{e}})^{\alpha}%
\]
Note also that $\prod_{(x,y)}[\frac{C_{x,y}^{\prime}}{C_{x,y}}]^{N_{x,y}%
}=\prod_{\{x,y\}}[\frac{C_{x,y}^{\prime}}{C_{x,y}}]^{N_{x,y}+N_{y,x}}$.

\bigskip

\begin{remark}
These $\frac{\mathcal{Z}_{e^{\prime}}}{\mathcal{Z}_{e}}$ determine, when
$e^{\prime}$ varies with $\frac{C^{^{\prime}}}{C}\leq1$ and $\frac{\lambda
^{\prime}}{\lambda}=1$, the Laplace transform of the distribution of the
traversal numbers of non oriented links $N_{x,y}+N_{y,x}$.
\end{remark}

\bigskip

Other variables of interest on the loop space are associated with elements of
the space $\mathbb{A}^{-}$ of\ odd functions $\omega$ on oriented links\ :
$\omega^{x,y}=-\omega^{y,x}$. Let us mention a few elementary results.

The operator $[P^{(\omega)}]_{y}^{x}=P_{y}^{x}\exp(i\omega^{x,y})$ is also
self adjoint in $L^{2}(\lambda)$. The associated loop variable writes
$\sum_{x,y}\omega^{x,y}N_{x,y}(l)$. We will denote it $\int_{l}\omega$. Note
it is invariant if $\omega^{x,y}$ is replaced by $\omega^{x,y}+g^{y}-g^{x}$
for some $g$. Set $[G^{(\omega)}]^{x,y}=\frac{[(I-P^{(\omega)})^{-1}]_{y}^{x}%
}{\lambda_{y}}$. By an argument similar to the one given above for the
occupation field, we have:

$\mathbb{P}_{x,x}^{t}(e^{i\int_{l}\omega}-1)=\exp(t(P^{(\omega)}%
-I))_{x,x}-\exp(t(P-I))_{x,x}$. Integrating in $t$ after expanding, we get
from the definition of $\mu$:%
\[
\int(e^{i\int_{l}\omega}-1)d\mu(l)=\sum_{k=1}^{\infty}\frac{1}{k}%
[Tr((P^{(\omega)})^{k})-Tr((P)^{k})]
\]
Hence%
\[
\int(e^{i\int_{l}\omega}-1)d\mu(l)=\log[\det(-L(I-P^{(\omega)})^{-1}]
\]
Hence $\int(e^{i\int_{l}\omega}-1)d\mu(l)=\log[\det(-L(I-P^{(\omega)})^{-1}]$
and
\[
\int(\exp(i\int_{l}\omega)-1)\mu(dl)=\log(\det(G^{(\omega)}G^{-1}))
\]
We can now extend the previous results (\ref{mupr}) and (\ref{F})\ to obtain,
setting $\det(G^{(\omega)})=\mathcal{Z}_{e,\omega}$%
\begin{equation}
\mu_{e}(e^{-\sum N_{x,y}\log(\frac{C_{x,y}^{^{\prime}}}{C_{x,y}})-\sum
(\lambda_{x}^{^{\prime}}-\lambda_{x})\widehat{l}_{x}+i\int_{l}\omega}%
-1)=\log(\frac{\mathcal{Z}_{e^{\prime},\omega}}{\mathcal{Z}_{e}}) \label{F4}%
\end{equation}
and%
\[
\mathbb{E}(\prod_{x,y}[\frac{C_{x,y}^{\prime}}{C_{x,y}}e^{i\omega_{x,y}%
}]^{N_{x,y}^{(\alpha)}}e^{-\sum(\lambda_{x}^{^{\prime}}-\lambda_{x}%
)\widehat{\mathcal{L}_{\alpha}}^{x}})=(\frac{\mathcal{Z}_{e^{\prime},\omega}%
}{\mathcal{Z}_{e}})^{\alpha}%
\]
Let us now introduce a new

\begin{definition}
We say that sets $\Lambda_{i}$ of non trivial loops are equivalent when the
associated occupation fields are equal and when the total traversal numbers
$\sum_{l\in\Lambda_{i}}N_{x,y}(l)$ are equal for all oriented edges $(x,y)$.
Equivalence classes will be called loop networks on the graph. We denote
$\overline{\Lambda}$ the loop network defined by $\Lambda$.

Similarly, a set $L$\ of non trivial discrete loops defines a discrete network
characterized by the total traversal numbers.\bigskip
\end{definition}

Note that these expectations determine the distribution of the network
$\overline{\mathcal{L}_{\alpha}}$\ defined by the loop ensemble $\mathcal{L}%
_{\alpha}$. We will denote $B^{e,e^{\prime},\omega}$ the variables
\[
\prod_{x,y}[\frac{C_{x,y}^{\prime}}{C_{x,y}}e^{i\omega_{x,y}}]^{N_{x,y}%
^{(\alpha)}}e^{-\sum(\lambda_{x}^{^{\prime}}-\lambda_{x})\widehat
{\mathcal{L}_{\alpha}}^{x}}.
\]

\begin{remark}
This last formula applies to the calculation of loop indices: If we have for
example a simple random walk on an oriented planar graph, and if $z^{\prime}$
is a point\ of the dual graph $X^{\prime}$, $\omega_{z^{\prime}}$ can be
chosen such that $\int_{l}\omega_{z^{\prime}}$\ is the winding number of the
loop around a given point $z^{\prime}$\ of the dual graph $X^{\prime}$. Then
$e^{i\pi\sum_{l\in\mathcal{L}_{\alpha}}\int_{l}\omega_{z}^{\prime}}$ is a spin
system of interest. We then get for example that%
\[
\mu(\int_{l}\omega_{z^{\prime}}\neq0)=-\frac{1}{2\pi}\int_{0}^{2\pi}\log
(\det(G^{(2\pi u\omega_{z^{\prime}})}G^{-1}))du
\]
and hence%
\[
\mathbb{P(}\sum_{l\in\mathcal{L}_{\alpha}}|\int_{l}\omega_{z^{\prime}%
}|)=0)=e^{\frac{\alpha}{2\pi}\int_{0}^{2\pi}\log(\det(G^{(2\pi u\omega
_{z^{\prime}})}G^{-1}))du}%
\]
Conditional distributions of the occupation field with respect to values of
the winding number can also be obtained.
\end{remark}

\section{Loop erasure and spanning trees.}

Recall that an oriented link $g$ is a pair of points $(g^{-},g^{+})$ such that
$C_{g}=C_{g^{-},g^{+}}\neq0$. Define $-g=(g^{+},g^{-})$.

Let $\mu_{x,y}^{\neq}$ be the measure induced by $C$ on discrete self-avoiding
paths between $x$ and $y$: $\mu_{\neq}^{x,y}(x,x_{2},...,x_{n-1}%
,y)=C_{x,x_{2}}C_{x_{1},x_{3}}...C_{x_{n-1},y}$.

Another way to defined a measure on discrete self avoiding paths from $x$ to
$y$ is loop erasure (see \cite{Law} ,\cite{Quianmp} and \cite{Law2}). In this
context, the loops can be trivial as they correspond to a single holding
times, and loop erasure produces a discrete path without holding times.

We have the following:

\begin{proposition}
The image of $\mu^{x,y}$ by the loop erasure map $\gamma\rightarrow\gamma
^{BE}$ is $\mu_{BE}^{x,y}$ defined on self avoiding paths by $\mu_{BE}%
^{x,y}(\eta)=\mu_{\neq}^{x,y}(\eta)\frac{\det(G)}{\det(G^{\{\eta\}^{c}})}%
=\mu_{\neq}^{x,y}(\eta)\det(G_{|\{\eta\}\times\{\eta\}})$ (Here $\{\eta\}$
denotes the set of points in the path $\eta$)
\end{proposition}

\begin{proof}
If $\eta=(x_{1}=x,x_{2},...x_{n}=y)$,and $\eta_{m}=(x,...x_{m})$,
\[
\mu^{x,y}(\gamma^{BE}=\eta)=\frac{\delta_{y}^{x}}{\lambda_{y}}+\sum
_{k=2}^{\infty}[P^{k}]_{x}^{x}P_{x_{2}}^{x}\mu_{\{x\}^{c}}^{x_{2},y}%
(\gamma^{BE}=\theta\eta)
\]
where $\mu_{\{x\}^{c}}^{x_{2},y}$ denotes the bridge measure for the Markov
chain killed as it hits $x$ and $\theta$ the natural shift on discrete paths.
By recurrence, this clearly equals
\[
V_{x}^{x}P_{x_{2}}^{x}[V^{\{x\}^{c}}]_{x_{2}}^{x_{2}}...[V^{\{\eta_{n-1}%
\}^{c}}]_{x_{n-1}}^{x_{n-1}}P_{y}^{x_{n-1}}[V^{\{\eta\}^{c}}]_{y}^{y}%
\lambda_{y}^{-1}=\mu_{\neq}^{x,y}(\eta)\frac{\det(G)}{\det(G^{\{\eta\}^{c}})}%
\]
as
\[
[V^{\{\eta_{m-1}\}^{c}}]_{x_{m}}^{x_{m}}=\frac{\det([(I-P]|_{\{\eta_{m}%
\}^{c}\times\{\eta_{m}\}^{c}})}{\det([(I-P]|_{\{\eta_{m-1}\}^{c}\times
\{\eta_{m-1}\}^{c}})}=\frac{\det(V^{\{\eta_{m-1}\}^{c}})}{\det(V^{\{\eta
_{m}\}^{c}})}=\frac{\det(G^{\{\eta_{m-1}\}^{c}})}{\det(G^{\{\eta_{m}\}^{c}}%
)}\lambda^{x_{m}}.
\]
for all $m\leq n-1$.
\end{proof}

\medskip

Also, by Feynman-Kac formula, for any self-avoiding path $\eta$:
\begin{align*}
\int e^{-<\widehat{\gamma},\chi>}1_{\{\gamma^{BE}=\eta\}}\mu^{x,y}(d\gamma)
&  =\frac{\det(G_{\chi})}{\det(G_{\chi}^{\{\eta\}^{c}})}\mu_{\neq}^{x,y}(\eta)
=\det(G_{\chi})_{|\{\eta\}\times\{\eta\}}\mu_{\neq}^{x,y}(\eta)\\
&  = \frac{\det(G_{\chi})_{|\{\eta\}\times\{\eta\}}}{\det(G_{|\{\eta
\}\times\{\eta\}})}\mu_{BE}^{x,y}(\eta).
\end{align*}

Therefore, recalling that by the results of section \ref{hit}\ conditionally
to $\eta$, $\widehat{\mathcal{L}}_{1}$ and $\widehat{\mathcal{L}}_{1}%
^{\{\eta\}^{c}}$ are independent, we see that under $\mu^{x,y}$, the
conditional distribution of $\widehat{\gamma}$ given $\gamma^{BE}=\eta$ is the
distribution of $\widehat{\mathcal{L}}_{1}-\widehat{\mathcal{L}}_{1}%
^{\{\eta\}^{c}}\mathcal{\ }$i.e. the occupation field of the loops of
$\mathcal{L}_{1}$\ which intersect $\eta$.

More generally, it can be shown that

\begin{proposition}
\label{be}The conditional distribution of the network $\overline
{\mathcal{L}_{\gamma}}$ defined by the loops of $\gamma$, given that
$\gamma^{BE}=\eta$, is identical to the distribution of the network defined by
$\mathcal{L}_{1}/\mathcal{L}_{1}^{\{\eta\}^{c}}$ i.e. the loops of
$\mathcal{L}_{1}$\ which intersect $\eta$.
\end{proposition}

\begin{proof}
Recall the notation $\mathcal{Z}_{e}=\det(G)$. First an elementary calculation
using (\ref{d})\ shows that $\mu_{e^{\prime}}^{x,y}(e^{i\int_{\gamma}\omega
}1_{\{\gamma^{BE}=\eta\}})$ equals
\begin{multline*}
\mu_{e}^{x,y}\Big(  1_{\{\gamma^{BE}=\eta\}}\prod[\frac{C_{\xi_{i},\xi_{i+1}%
}^{\prime}}{C_{\xi_{i},\xi_{i+1}}}e^{i\omega_{\xi_{i},\xi_{i+1}}}%
\frac{\lambda_{\xi_{i}}}{\lambda_{\xi_{i}}^{\prime}}]\Big) \\
\frac{C_{x,x_{2}}^{\prime}C_{x_{1},x_{3}}^{\prime}...C_{x_{n-1},y}^{\prime}%
}{C_{x,x_{2}}C_{x_{1},x_{3}}...C_{x_{n-1},y}}e^{i\int_{\eta}\omega}\mu
_{e}^{x,y}\Big(  \prod_{u\neq v}[\frac{C_{u,v}^{\prime}}{C_{u,v}}%
e^{i\omega_{u,v}}]^{N_{u,v}(\mathcal{L}_{\gamma})}e^{-\left\langle
\lambda^{^{\prime}}-\lambda,\widehat{\mathcal{\gamma}}\right\rangle
}1_{\{\gamma^{BE}=\eta\}}\Big)  .
\end{multline*}

(Note the term $e^{-\left\langle \lambda^{^{\prime}}-\lambda,\widehat
{\mathcal{\gamma}}\right\rangle }$ can be replaced by $\prod_{u}%
(\frac{\lambda_{u}}{\lambda_{u}^{\prime}})^{N_{u}(\gamma)}$).

Moreover, by the proof of the previous proposition, applied to  the Markov
chain defined by $e^{\prime}$ perturbed by $\omega$, we have also
\[\mu_{e^{\prime}}^{x,y}(e^{i\int_{\gamma}\omega}1_{\{\gamma^{BE}=\eta
\}})=C_{x,x_{2}}^{\prime}C_{x_{1},x_{3}}^{\prime}...C_{x_{n-1},y}^{\prime
}e^{i\int_{\eta}\omega}\frac{\mathcal{Z}_{[e^{\prime}]^{\{\eta\}^{c}},\omega}%
}{\mathcal{Z}_{e^{\prime},\omega}}.\]

Therefore%
\[
\mu_{e}^{x,y}(\prod_{u\neq v}[\frac{C_{u,v}^{\prime}}{C_{u,v}}e^{i\omega
_{u,v}}]^{N_{u,v}(\mathcal{L}_{\gamma})}e^{-\left\langle \lambda^{^{\prime}%
}-\lambda,\widehat{\mathcal{\gamma}}\right\rangle }||\gamma^{BE}%
=\eta)=\frac{\mathcal{Z}_{e}\mathcal{Z}_{[e^{\prime}]^{\{\eta\}^{c}},\omega}%
}{\mathcal{Z}_{e^{\{\eta\}^{c}}}\mathcal{Z}_{e^{\prime},\omega}}.
\]
Moreover, by (\ref{F4}) and the properties of the Poisson processes,%

\[
\mathbb{E}(\prod_{u\neq v}[\frac{C_{u,v}^{\prime}}{C_{u,v}}e^{i\omega_{u,v}%
}]^{N_{u,v}(\mathcal{L}_{1}/\mathcal{L}_{1}^{\{\eta\}^{c}})}e^{-\left\langle
\lambda^{^{\prime}}-\lambda,\widehat{\mathcal{L}}_{1}-\widehat{\mathcal{L}%
}_{1}^{\{\eta\}^{c}}\right\rangle }=\frac{\mathcal{Z}_{e}\mathcal{Z}%
_{[e^{\prime}]^{\{\eta\}^{c}},\omega}}{\mathcal{Z}_{e^{\{\eta\}^{c}}%
}\mathcal{Z}_{e^{\prime},\omega}}.
\]

It follows that the joint distribution of the traversal numbers and the
occupation field are identical for the set of erased loops and $\mathcal{L}%
_{1}/\mathcal{L}_{1}^{\{\eta\}^{c}}$.
\end{proof}

\bigskip

Similarly one can define the image of $\mathbb{P}^{x}$ by $BE$ which is given
by
\[
\mathbb{P}_{BE}^{x}(\eta)=C_{x_{1},x_{2}}...C_{x_{n-1},x_{n}}\kappa_{x_{n}%
}\det(G_{|\{\eta\}\times\{\eta\}}),
\]
for $\eta=(x_{1},...,x_{n})$, and get the same results.

\bigskip

Wilson's algorithm (see \cite{Lyo2}) iterates this construction, starting with
$x^{\prime}s$ in arbitrary order. Each step of the algorithm reproduces the
first step except it stops when it hits the already constructed tree of self
avoiding paths. It provides a construction of a random spanning tree. Its law
is a probability measure $\mathbb{P}_{ST}^{e}$\ on the set $ST_{X,\Delta}$\ of
spanning trees of $X$ rooted at the cemetery point $\Delta$ defined by the
energy $e$. The weight attached to each oriented link $g=(x,y)$ of $X\times X$
is the conductance and the weight attached to the link $(x,\Delta)$ is
$\kappa_{x}$ we can also denote by $C_{x,\Delta}$. As the determinants
simplify, the probability of a tree $\Upsilon$ is given by a simple formula:%

\begin{equation}
\mathbb{P}_{ST}^{e}(\Upsilon)=\mathcal{Z}_{e}\prod_{\xi\in\Upsilon}C_{\xi}
\label{span}%
\end{equation}
It is clearly independent of the ordering chosen initially. Now note that,
since we get a probability
\[
\mathcal{Z}_{e}\sum_{\Upsilon\in ST_{X,\Delta}}\prod_{(x,y)\in\Upsilon}%
C_{x,y}\prod_{x,(x,\Delta)\in\Upsilon}\kappa_{x}=1
\]
or equivalently
\[
\sum_{\Upsilon\in ST_{X,\Delta}}\prod_{(x,y)\in\Upsilon}P_{y}^{x}%
\prod_{x,(x,\Delta)\in\Upsilon}P_{\Delta}^{x}=\frac{1}{\prod_{x\in X}%
\lambda_{x}\mathcal{Z}_{e}}%
\]
Then, it comes that, for any $e^{\prime}$ for which conductances (including
$\kappa^{\prime}$) are positive only on links of $e$,%
\[
\mathbb{E}_{ST}^{e}\left(  \prod_{(x,y)\in\Upsilon}\frac{P_{y}^{\prime x}%
}{P_{y}^{x}}\prod_{x,(x,\Delta)\in\Upsilon}\frac{P_{\Delta}^{\prime x}%
}{P_{\Delta}^{x}}\right)  =\frac{\prod_{x\in X}\lambda_{x}}{\prod_{x\in
X}\lambda_{x}^{\prime}}\frac{\mathcal{Z}_{e}}{\mathcal{Z}_{e^{\prime}}}%
\]
and%
\begin{equation}
\mathbb{E}_{ST}^{e}\left(  \prod_{(x,y)\in\Upsilon}\frac{C_{x,y}^{\prime}%
}{C_{x,y}}\prod_{x,(x,\Delta)\in\Upsilon}\frac{\kappa_{x}^{\prime}}{\kappa
_{x}}\right)  =\frac{\mathcal{Z}_{e}}{\mathcal{Z}_{e^{\prime}}} \label{rel}%
\end{equation}

Note also that in the case of a graph (i.e. when all conductances are equal to
$1$), all spanning trees have the same probability. The expression of their
cardinal as the determinant $\mathcal{Z}_{e}$\ is Cayley's theorem (see for
exemple \cite{Lyo2}).

\begin{corollary}
The network defined by the random set of loops $\mathcal{L}_{W}$\ constructed
in this algorithm is independent of the random spanning tree, and independent
of the ordering. It has the same distribution as the network defined by the
loops of $\mathcal{L}_{1}$.
\end{corollary}

This result follows easily from proposition \ref{be}.

\bigskip

\section{Decompositions}

Note first that with the energy $e$, we can associate a rescaled Markov chain
$\widehat{x}_{t}$\ in which holding times at any point $x$ are exponential
times of parameters $\lambda_{x}$: $\widehat{x}_{t}=x_{\tau_{t}}$ with
$\tau_{t}=\inf(s,\;\int_{0}^{s}\frac{1}{\lambda_{x_{u}}}du=t)$. For the
rescaled Markov chain, local times coincide with the time spent in a point and
the duality measure is simply the counting measure. The Markov loops can be
rescaled as well and we did it in fact already when we introduced pointed
loops. More generally we may introduce different holding times parameters but
it would be essentially useless as the random variables we are interested into
are intrinsic, i.e. depend only on $e$.

If $D\subset X$ and we set $F=D^{c}$, the orthogonal decomposition of the
energy $e(f,f)=e(f)$\ into $e^{D}(f-H^{F}f)+e(H^{F}f)$ leads to the
decomposition of the Gaussian field mentioned above and also to a
decomposition of the rescaled Markov chain into the rescaled Markov chain
killed at the exit of $D$ and the trace of the rescaled Markov chain on $F$,
i.e. $\widehat{x}_{t}^{\{F\}}=\widehat{x}_{S_{t}^{F}}$, with $S_{t}^{F}%
=\inf(s,\int_{0}^{s}1_{F}(\widehat{x}_{u})du=t)$.

\begin{proposition}
The trace of the rescaled Markov chain on $F$ is the rescaled Markov chain
defined by the energy functional $e^{\{F\}}(f)=e(H^{F}f)$ , for which
\[
C_{x,y}^{\{F\}}=C_{x,y}+\sum_{a,b\in D}C_{x,a}C_{b,y}[G^{D}]^{a,b}%
\]%
\[
\lambda_{x}^{\{F\}}=\lambda_{x}-\sum_{a,b\in D}C_{x,a}C_{b,x}[G^{D}]^{a,b}%
\]
and%
\[
\mathcal{Z}_{e}=\mathcal{Z}_{e^{D}}\mathcal{Z}_{e^{\{F\}}}%
\]
\end{proposition}

\begin{proof}
For the second assertion, note first that for any $y\in F$,
\[
\lbrack H^{F}]_{y}^{x}=1_{x=y}+1_{D}(x)\sum_{b\in D}[G^{D}]^{x,b}C_{b,y}.
\]
Moreover, $e(H^{F}f)=e(f,H^{F}f)$ and therefore
\[
\lambda_{x}^{\{F\}}=e^{\{F\}}(1_{\{x\}})=e(1_{\{x\}},H^{F}1_{\{x\}}%
)=\lambda_{x}-\sum_{a\in D}C_{x,a}[H^{F}]_{x}^{a}=\lambda_{x}(1-p_{x}%
^{\{F\}})
\]
where $p_{x}^{\{F\}}=\sum_{a,b\in D}P_{a}^{x}[G^{D}]^{a,b}C_{b,x}=\sum_{a\in
D}P_{a}^{x}[H^{F}]_{x}^{a}$ is the probability that the Markov chain starting
at $x$ will return to $x$ after an excursion in $D$.\newline Then for distinct
$x$ and $y$ in $F$,
\begin{align*}
C_{x,y}^{\{F\}}  &  =-e^{\{F\}}(1_{\{x\}},1_{\{y\}})=-e(1_{\{x\}}%
,H^{F}1_{\{y\}})\\
&  =C_{x,y}+\sum_{a}C_{x,a}[H^{F}]_{y}^{a}=C_{x,y}+\sum_{a,b\in D}%
C_{x,a}C_{b,y}[G^{D}]^{a,b}.
\end{align*}

Note that the graph defined on $F$ by the non vanishing conductances
$C_{x,y}^{\{F\}}$ has in general more edges than the restiction to $F$ of the
original graph.

For the third assertion, note also that $G^{\{F\}}$ is the restriction of $G$
to $F$ as for all $x,y\in F$, $e^{\{F\}}(G\delta_{y|F},1_{\{x\}}%
)=e(G\delta_{y},[H^{F}1_{\{x\}}])=1_{\{x=y\}}$. Hence the determinant
decomposition already used in section \ref{hit} yields the final formula. The
cases where $F$ has one point was already treated in section \ref{hit}.

Finally, for the first assertion note the transition matrix $[P^{\{F\}}%
]_{y}^{x}$ can be computed directly and equals

$P_{y}^{x}+$ $\sum_{a,b\in D}P_{a}^{x}P_{y}^{b}V^{D\cup\{x\}}]_{b}^{a}%
=P_{y}^{x}+$ $\sum_{a,b\in D}P_{a}^{x}C_{b,y}[G^{D\cup\{x\}}]^{a,b}$. It can
be decomposed according whether the jump to $y$ occurs from $x$ or from $D$
and the number of excursions from $x$ to $x$:
\begin{align*}
[P^{\{F\}}]_{y}^{x}  &  =\sum_{k=0}^{\infty}(\sum_{a,b\in D}P_{a}^{x}%
[V^{D}]_{b}^{a}P_{x}^{b})^{k}(P_{y}^{x}+\sum_{a,b\in D}P_{a}^{x}[V^{D}%
]_{b}^{a}P_{y}^{b})\\
&  =\sum_{k=0}^{\infty}(\sum_{a,b\in D}P_{a}^{x}[G^{D}]^{a,b}C_{b,x}%
)^{k}(P_{y}^{x}+\sum_{a,b\in D}P_{a}^{x}[G^{D}]^{a,b}C_{b,y}).
\end{align*}
The expansion of $\frac{C_{x,y}^{\{F\}}}{\lambda_{x}^{\{F\}}}$ in geometric
series yields the exactly the same result.

Finally, remark that the holding times of $\widehat{x}_{t}^{\{F\}}$ at any
point $x\in F$\ are sums of a random number of independent holding times of
$\widehat{x}_{t}$. This random integer counts the excursions from $x$ to $x$
performed by the chain $\widehat{x}_{t}$ during the holding time of
$\widehat{x}_{t}^{\{F\}}$. It follows a geometric distribution of parameter
$1-p_{x}^{\{F\}}$. Therefore, $\frac{1}{\lambda_{x}^{\{F\}}}=\frac{1}%
{\lambda_{x}(1-p_{x})}$ is the expectation of the holding times of
$\widehat{x}_{t}^{\{F\}}$ at $x$.
\end{proof}

\medskip

If $\chi$ is carried by $D$ and if we set $e_{\chi}=e+\left\|  \quad\right\|
_{L^{2}(\chi)}$ and denote $[e_{\chi}]^{\{F\}}$ by $e^{\{F,\chi\}}$ we have
\[
C_{x,y}^{\{F,\chi\}}=C_{x,y}+\sum_{a,b}C_{x,a}C_{b,y}[G_{\chi}^{D}%
]^{a,b},\quad p_{x}^{\{F,\chi\}}=\sum_{a,b\in D}P_{a}^{x}[G_{\chi}^{D}%
]^{a,b}C_{b,x}%
\]
and $\lambda_{x}^{\{F,\chi\}}=\lambda_{x}(1-p_{x}^{\{F,\chi\}})$.\newline More
generally, if $e^{\#}$ is such that $C^{\#}=C$\ on $F\times F$, and
$\lambda=\lambda^{\#}$ on $F$ we have:
\[
C_{x,y}^{\#\{F\}}=C_{x,y}+\sum_{a,b}C_{x,a}^{\#}C_{b,y}^{\#}[G^{\#D}%
]^{a,b},\quad p_{x}^{\#\{F\}}=\sum_{a,b\in D}P_{a}^{\#x}[G^{\#D}]^{a,b}C_{b,x}%
\]
and $\lambda_{x}^{\#\{F\}}=\lambda_{x}(1-p_{x}^{\#\{F\}}).$\newline

A loop in $X$ which hits $F$\ can be decomposed into a loop $l^{\{F\}}$\ in
$F$ and its excursions in $D$ which may come back to their starting point. Let
$\mu_{D}^{a,b}$\ denote the bridge measure (with mass $[G^{D}]^{a,b}%
$)\ associated with $e^{D}$.\newline Set
\[
\nu_{x,y}^{D}=\frac{1}{C_{x,y}^{\{F\}}}[C_{x,y}\delta_{\emptyset}+\sum_{a,b\in
D}C_{x,a}C_{b,y}\mu_{D}^{a,b}],\quad\rho_{x}^{D}=\sum_{n=1}^{\infty}%
\frac{1}{\lambda_{x}p_{x}^{\{F\}}}(\sum_{a,b\in D}C_{x,a}C_{b,x}\mu_{D}%
^{a,b})
\]
and $\nu_{x}^{D}=\frac{1}{1-p_{x}^{\{F\}}}[\delta_{\emptyset}+\sum
_{n=1}^{\infty}[p_{x}^{\{F\}}\rho_{x}^{D}]^{\otimes n}]$.\newline Note that
$\rho_{x}^{D}(1)=\nu_{x,y}^{D}(1)=\nu_{x}^{D}(1)=1$.

A loop $l$ can be decomposed into its restriction $l^{\{F\}}=(\xi_{i}%
,\widehat{\tau}_{i})$\ in $F$ (possibly a one point loop), a family of
excursions $\gamma_{\xi_{i},\xi_{i+1}}$\ attached to the jumps of $l^{\{F\}}$
and systems of i.i.d. excursions $(\gamma_{\xi_{i}}^{h},h\leq n_{\xi_{i}})$
attached to the points of $l^{\{F\}}$. Note the set of excursions can be empty.

We get a decomposition of $\mu$ into its restriction $\mu^{D}$\ to loops in
$D$ (associated to the process killed at the exit of $D$), the loop measure
$\mu^{\{F\}}$\ defined on loops of $F$ by the trace of the Markov chain on
$F$, probability measures $\nu_{x,y}^{D}$\ on excursions in $D$ indexed by
pairs of points in $F$ and measures $\rho_{x}^{D}\ $on excursions in
$D$\ indexed by points of $F$. Moreover, the integers $n_{\xi_{i}}$ follow a
Poisson distribution of parameter $\lambda_{\xi_{i}}^{\{F\}}\widehat{\tau}%
_{i}$ and the conditional distribution of the rescaled holding times in
$\xi_{i}$ before each excursion $\gamma_{\xi_{i}}^{l}$ is the distribution
$\beta_{n_{\xi_{i}},\tau_{i}^{\ast}}$\ of the increments of a uniform sample
of $n_{\xi_{i}}$ points in $[0\;\widehat{\tau}_{i}]$ put in increasing order.
We denote these holding times by $\widehat{\tau}_{i,h}$ and set $l=\Lambda
(l^{\{F\}},(\gamma_{\xi_{i},\xi_{i+1}}),(n_{\xi_{i}},\gamma_{\xi_{i}}%
^{h},\widehat{\tau}_{i,h}))$.

Then $\mu-\mu^{D}$ is the image measure by $\Lambda$ of
\[
\mu^{\{F\}}(dl^{\{F\}})\prod(\nu_{\xi_{i},\xi_{i+1}}^{D})(d\gamma_{\xi_{i}%
,\xi_{i+1}})\prod e^{-\lambda_{\xi_{i}}^{\{F\}}\widehat{\tau}_{i}}%
\sum\frac{[\lambda_{\xi_{i}}^{\{F\}}\widehat{\tau}_{i}]^{k}}{k!}1_{n_{\xi_{i}%
}=k}[\rho_{x}^{D}]^{\otimes k}(d\gamma_{\xi_{i}}^{h})\beta_{k,\tau_{i}^{\ast}%
}(d\widehat{\tau}_{i,h}).
\]

The Poisson process $\mathcal{L}_{\alpha}^{\{F\}}=\{l^{\{F\}},l\in
\mathcal{L}_{\alpha}\}$ has intensity $\mu^{\{F\}}$ and is independent of
$\mathcal{L}_{\alpha}^{D}$.

Note that $\widehat{\mathcal{L}_{\alpha}^{\{F\}}}$ is the restriction of
$\widehat{\mathcal{L}_{\alpha}}$ to $F$.\newline In particular, if $\chi$ is a
measure carried by $D$, we have:%
\begin{align*}
\mathbb{E}(e^{-\left\langle \widehat{\mathcal{L}_{\alpha}},\chi\right\rangle
}|\mathcal{L}_{\alpha}^{\{F\}})  &  =\mathbb{E}(e^{-\left\langle
\widehat{\mathcal{L}_{\alpha}^{D}},\chi\right\rangle })(\prod_{x,y\in F}[\int
e^{-\left\langle \widehat{\mathcal{\gamma}},\chi\right\rangle }\nu_{x,y}%
^{D}(d\gamma)]^{N_{x,y}(\mathcal{L}_{\alpha}^{\{F\}})}\\
&  \times\prod_{x\in F}e^{\lambda_{x}^{\{F\}}[\widehat{\mathcal{L}_{\alpha
}^{\{F\}}}]^{x}\int(e^{-\left\langle \widehat{\mathcal{\gamma}},\chi
\right\rangle }-1)\rho_{x}^{D}(d\gamma)}\\
&  =[\frac{\mathcal{Z}_{e_{\chi}^{D}}}{\mathcal{Z}_{e^{D}}}]^{\alpha}%
(\prod_{x,y\in F}[\frac{C_{x,y}^{\{F,\chi\}}}{C_{x,y}^{\{F\}}}]^{N_{x,y}%
(\mathcal{L}_{\alpha}^{\{F\}})}\prod_{x\in F}e^{[\lambda_{x}^{\{F,\chi
\}}-\lambda_{x}^{\{F\}}]\widehat{\mathcal{L}_{\alpha}^{x}}}.
\end{align*}
(recall that $\widehat{\mathcal{L}_{\alpha}^{\{F\}}}$ is the restriction of
$\widehat{\mathcal{L}_{\alpha}}$ to $F$). Also, if we condition on the set of
discrete loops $\mathcal{DL}_{\alpha}^{\{F\}}$%
\[
\mathbb{E}(e^{-\left\langle \widehat{\mathcal{L}_{\alpha}},\chi\right\rangle
}|\mathcal{DL}_{\alpha}^{\{F\}})=[\frac{\mathcal{Z}_{e_{\chi}^{D}}%
}{\mathcal{Z}_{e^{D}}}]^{\alpha}(\prod_{x,y\in F}[\frac{C_{x,y}^{\{F,\chi\}}%
}{C_{x,y}^{\{F\}}}]^{N_{x,y}(\mathcal{L}_{\alpha}^{\{F\}})}\prod_{x\in
F}[\frac{\lambda_{x}^{\{F\}}}{\lambda_{x}^{\{F,\chi\}}}]^{N_{x}(\mathcal{L}%
_{\alpha}^{\{F\}})+1})
\]
where the last exponent $N_{x}+1$ is obtained by taking into account the loops
which have a trivial trace on $F$ (see formula (\ref{unpt})).

More generally we can show in the same way the following

\begin{proposition}
If $C^{\#}=C$\ on $F\times F$, and $\lambda=\lambda^{\#}$ on $F$, we denote
$B^{e,e^{\#}}$ the multiplicative functional $\displaystyle   {\prod
_{x,y}[\frac{C_{x,y}^{\#}}{C_{x,y}}]^{N_{x,y}}e^{-\sum_{x\in D}\widehat{l_{x}%
}(\lambda_{x}^{\#}-\lambda_{x})}}$.\newline Then,
\[
\mathbb{E}(B^{e,e^{\#}}|\mathcal{L}_{\alpha}^{\{F\}})=[\frac{\mathcal{Z}%
_{e^{\#D}}}{\mathcal{Z}_{e^{D}}}]^{\alpha}(\prod_{x,y\in F}[\frac{C_{x,y}%
^{\#\{F\}}}{C_{x,y}^{\{F\}}}]^{N_{x,y}(\mathcal{L}_{\alpha}^{\{F\}})}%
\prod_{x\in F}e^{\lambda_{x}[p_{x}^{\#\{F\}}-p_{x}^{\{F\}}]\widehat
{\mathcal{L}_{\alpha}^{x}}}%
\]
and
\[
\mathbb{E}(B^{e,e^{\#}}|\mathcal{DL}_{\alpha}^{\{F\}})=[\frac{\mathcal{Z}%
_{e^{\#D}}}{\mathcal{Z}_{e^{D}}}]^{\alpha}(\prod_{x,y\in F}[\frac{C_{x,y}%
^{\#\{F\}}}{C_{x,y}^{\{F\}}}]^{N_{x,y}(\mathcal{L}_{\alpha}^{\{F\}})}%
\prod_{x\in F}[\frac{\lambda_{x}^{\{F\}}}{\lambda_{x}^{\#\{F\}}}%
]^{N_{x}(\mathcal{L}_{\alpha}^{\{F\}})+1}%
\]
\end{proposition}

\medskip These decomposition and conditional expectation formulas extend to
include a current $\omega$. Note that $e^{\{F\}}$ will depend on $\omega$
unless it is closed (i.e. vanish on every loop) in $D$. In particular, it
allows to define $\omega^{F}$ such that:%
\[
\mathcal{Z}_{e,\omega}=\mathcal{Z}_{e^{D}}\mathcal{Z}_{e^{\{F\}},\omega^{F}}%
\]
The previous proposition implies the following \textsl{Markov property}:

\begin{remark}
If $D=D_{1}\cup D_{2}$ with $D_{1}$ and $D_{2}$ stongly disconnected, (i.e.
such that for any $(x,y,z)\in D_{1}\times D_{2}\times F$, $C_{x,y}$ and
$C_{x,z}C_{y,z}$ vanish), the restrictions of the network $\overline
{\mathcal{L}_{\alpha}}$ to $D_{1}\cup F$ and $D_{2}\cup F$ are independent
conditionally to the restriction of $\mathcal{L}_{\alpha}$\ to $F$.
\end{remark}

\begin{proof}
It follows from the fact that as $D_{1}$ and $D_{2}$\ are disconnected, any
excursion measure $\nu_{x,y}^{D}$ or $\rho_{x}^{D}$\ from $F$ into
$D=D_{1}\cup D_{2}$ is an excursion measure either in $D_{1}$ or in $D_{2}$.
\end{proof}

\bigskip

\bigskip

\paragraph{Branching processes with immigration}

An interesting example can be given after extending slightly the scope of the
theory to countable transient symmetric Markov chains: We can take
$X=\mathbb{N}-\{0\}$, $C_{n,n+1}=1$ for all $n\geq1$ and $\kappa_{1}=1$ and
$P$ to be the transfer matrix of the simple symmetric random walk killed at
$0$.

Then we can apply the previous considerations to check that $\widehat
{\mathcal{L}}_{\alpha}^{n}$ is a branching process with immigration.

The immigration at level $n$ comes from the loops whose infimum is $n$ and the
branching from the excursions of the loops existing at level $n$ to level
$n+1$. Set $F_{n}=\{1,2...n\}$ and $D_{n}=F_{n}^{c}$.

The immigration law (on $\mathbb{R}^{+}$) is a Gamma distribution
$\Gamma(\alpha,G^{1,1})$. It is the law of $\widehat{\mathcal{L}}_{\alpha}%
^{1}$ and also of $[\widehat{\mathcal{L}}_{\alpha}^{D_{n-1}}]^{n}$\ for all
$n>1$. From the above calculations of conditional expectations, we get that
for any positive parameter $\gamma$,%

\[
\mathbb{E}(e^{-[\gamma\mathcal{L}_{\alpha}^{n}\widehat{\mathcal{L}}_{\alpha
}^{n}}||\mathcal{L}_{\alpha}^{\{F_{n-1}\}})=\mathbb{E}(e^{-[\gamma
\widehat{\mathcal{L}}_{\alpha}^{D_{n-1}}]^{n}})e^{\lambda_{n-1}^{\{F_{n-1}%
,\gamma\delta_{n}\}}-\lambda_{n-1}^{\{F_{n-1}\}}]\widehat{\mathcal{L}}%
_{\alpha}^{n-1}}%
\]
From this formula, it is clear that $\widehat{\mathcal{L}}_{\alpha}^{n}$ is a
Markov process. To be more precise, note that for any $n,m>0$, $V_{m}%
^{n}=2(n\wedge m)$ and $\lambda_{n}=2$, that $G_{\gamma\delta_{1}}%
^{1,n}=G^{1,n}-G^{1,1}\gamma G_{\gamma\delta_{1}}^{1,n}$ so that
$G_{\gamma\delta_{1}}^{1,n}=\frac{1}{1+\gamma}$and that for any $n>0$, the
restriction of the Markov chain to $D_{n}$ is isomorphic to the original
Markov chain. Then it comes that for all $n$, $p_{n}^{\{F_{n}\}}=\frac{1}{2}$,
$\lambda_{n}^{\{F_{n}\}}=1$, $p_{n}^{\{F_{n},\gamma\delta_{n+1}\}}%
=\frac{1}{2(1+\gamma)}$ and $\lambda_{n}^{\{F_{n},\gamma\delta_{n+1}%
\}}=\frac{2\gamma+1}{1+\gamma}$\ so that the Laplace exponent of the
convolution semigroup $\nu_{t}$\ defining the branching mechanism equals
$\frac{\gamma}{1+\gamma}=\int(1-e^{-\gamma s})e^{-s}ds$. It is the semigroup
of a compound Poisson process whose Levy measure is exponential. The
conditional law of $\widehat{\mathcal{L}}_{\alpha}^{n+1}$ given $\widehat
{\mathcal{L}}_{\alpha}^{n}$ is the convolution of the immigration law
$\Gamma(\alpha,1)$ with $\nu_{\widehat{\mathcal{L}}_{\alpha}^{n}}$.

Alternatively, we can consider the integer valed process $N_{n}(\mathcal{L}%
_{\alpha}^{\{F_{n}\}})+1$ which is a Galton Watson process with immigration.
In our exemple, we find the reproduction law $\pi(n)=2^{-n-1}$for all $n\geq0$
(critical binary branching).

If we consider the occupation field defined by the loops going through $1$, we
get a branching process without immigration: it is the classical relation
between random walks local times and branching processes.

\bigskip

\section{The case of general Markov processes}

We now explain briefly how some of the above results will be extended to a
symmetric Markov process on an infinite space $X$. The construction of the
loop measure as well as a lot of computations can be performed quite
generally, using Markov processes or Dirichlet space theory (Cf for example
\cite{Fukutak}). It works as soon as the bridge or excursion measures
$\mathbb{P}_{t}^{x,y}$\ can be properly defined. The semigroup should have a
locally integrable kernel $p_{t}(x,y)$.

\bigskip

Let us consider more closely the occupation field $\widehat{l}$. \ The
extension is rather straightforward when points are not polar. We can start
with a Dirichlet space of continuous functions and a measure $m$ such that
there is a mass gap. Let $P_{t}$ the associated Feller semigroup. Then the
Green function is well defined as the mutual energy of the Dirac measures
$\delta_{x}$ and $\delta_{y}$ which have finite energy. It is the covariance
function of a Gaussian free field $\phi(x)$, which will be associated to the
field $\widehat{\mathcal{L}}_{\frac{1}{2}}^{x}$ of local times of the Poisson
process of random loops whose intensity is given by the loop measure defined
by the semigroup $P_{t}$. This will apply to examples related to one
dimensional Brownian motion or to Markov chains on countable spaces.

When we consider Brownian motion on the half line, we get a continuous
branching process with immigration, as in the discrete case.

When points are polar, one needs to be more careful. We will consider only the
case of the two and three dimensional Brownian motion in a bounded domain
$D$\ killed at the boundary, i.e. associated with the classical energy with
Dirichlet boundary condition. The Green function does not induce a trace class
operator but it is still Hilbert-Schmidt which allows to define renormalized
determinants $\det_{2}$ (Cf \cite{Sim}).

If $A$ is a symmetric Hilbert Schmidt operator, $\det_{2}(I+A)$ is defined as
$\prod(1+\lambda_{i})e^{-\lambda_{i}}$ where $\lambda_{i}$ are the eigenvalues
of $A$.

The Gaussian field (called free field) whose covariance function is the Green
function is now a generalized field: Generalized fields are not defined
pointwise but have to be smeared by a test function $f$. Still $\phi(f)$ is
often denoted $\int\phi(x)f(x)dx.$

Wick powers $:\phi^{n}:$\ of the free field can be defined as generalized
field by approximation as soon as the $2n$-th power of the Green function,
$G(x,y)^{2n}$ is locally integrable (Cf \cite{Sim2}). This is the case for all
$n$ for Brownian motion in dimension two, as the Green function has only a
logarithmic singularity on the diagonal, and for $n=2$ in dimension three as
the singularity is of the order of $\frac{1}{\left\|  x-y\right\|  }$. More
precisely, taking for example $\pi_{\varepsilon}^{x}(dy)$ to be the normalized
area measure on the sphere of radius $\varepsilon$ around $x$, $\phi
(\pi_{\varepsilon}^{x})$ is a Gaussian field with covariance $\sigma
_{\varepsilon}^{x}=\int G(z,z^{\prime})\pi_{\varepsilon}^{x}(dz)\pi
_{\varepsilon}^{y}(dz^{\prime})$. Its Wick powers are defined with Hermite
polynomials as we did previously:

$:\phi(\pi_{\varepsilon}^{x})^{n}:=(\sigma_{\varepsilon}^{x})^{\frac{n}{2}%
}H_{n}(\frac{\phi(\pi_{\varepsilon}^{x})}{\sqrt{\sigma_{\varepsilon}^{x}}})$.
Then one can see that, $\int f(x):\phi(\pi_{\varepsilon}^{x})^{n}:dx$
converges in $L^{2}$ for any bounded continuous function $f$ with compact
support towards a limit called the $n$-th Wick power of the free field
evaluated on $f$ and denoted $:\phi^{n}:(f)$. Moreover, $\mathbb{E}(:\phi
^{n}:(f):\phi^{n}:(h))=\int G^{2n}(x,y)f(x)h(y)dxdy$.

In these cases, we can extend the statement of theorem \ref{iso} to the
renormalized occupation field $\widetilde{\mathcal{L}}_{\frac{1}{2}}^{x}$\ and
the Wick square $:\phi^{2}:$\ of the free field.

\bigskip

Let us explain this in more details in the Brownian motion case. Let $D$ be an
open subset of $\mathbb{R}^{d}$ such that the Brownian motion killed at the
boundary of $D$ is transient and has a Green function.\ Let $p_{t}(x,y)$ be
its transition density and $G(x,y)=\int_{0}^{\infty}p_{t}(x,y)dt$\ the
associated Green function. The loop measure $\mu$ was defined in \cite{LW} as
\[
\mu=\int_{D}\int_{0}^{\infty}\frac{1}{t}\mathbb{P}_{t}^{x,x}dt
\]
where $\mathbb{P}_{t}^{x,x}$ denotes the (non normalized) excursion
measure\ of duration $t$ such that if $0\leq t_{1}\leq...t_{h}\leq t$,%
\[
\mathbb{P}_{t}^{x,x}(\xi(t_{1})\in dx_{1},...,\xi(t_{h})\in dx_{h})=p_{t_{1}%
}(x,x_{1})p_{t_{2}-t_{1}}(x_{1},x_{2}).......p_{t-t_{h}}(x_{h},x)dx_{1}%
...dx_{h}%
\]
(the mass of $\mathbb{P}_{t}^{x,x}$ is $p_{t}(x,x)$). Note that $\mu$ is a
priori defined on based loops but it is easily seen to be shift-invariant.

\bigskip

For any loop $l$ indexed by $[0\;T(l)]$, define the measure $\widehat{l}%
=\int_{0}^{T(l)}\delta_{l(s)}ds$: for any Borel set $A$, $\widehat{l}%
(A)=\int_{0}^{T(l)}1_{A}(l_{s})ds$. As before, we have the following:

\begin{lemma}
For any non negative function $f$,
\[
\mu(\left\langle \widehat{l},f\right\rangle ^{n})=(n-1)!\int G(x_{1}%
,x_{2})f(x_{2})G(x_{2},x_{3})f(x_{3})...G(x_{n},x_{1})f(x_{1})\prod_{1}%
^{n}dx_{i}%
\]
\end{lemma}

One can define in a similar way the analogous of multiple local times, and get
for their integrals with respect to $\mu$ a formula analogous to the one
obtained in the discrete case.\bigskip

Let $G$ denote the operator on $L^{2}(D,dx)$ defined by $G$. Let $f$ be a non
negative continuous function with compact support in $D$.

Note that $\left\langle \widehat{l},f\right\rangle $ is $\mu$-integrable only
in dimension one as then, $G$\ is locally trace class. In that case, using for
all $x$ an approximation of the Dirac measure at $x$, local times $\widehat
{l}^{x}$ can be defined in such a way that $\left\langle \widehat
{l},f\right\rangle =\int\widehat{l}^{x}f(x)dx$.

$\left\langle \widehat{l},f\right\rangle $ is $\mu$-square integrable in
dimensions one, two and three, as $G$\ is Hilbert-Schmidt if $D$ is bounded,
since $\int\int_{D\times D}G(x,y)^{2}dxdy<\infty$, and otherwise locally
Hilbert-Schmidt.\bigskip\newline \textbf{N.B.:} Considering distributions
$\chi$ such that $\int\int(G(x,y)^{2}\chi(dx)\chi(dy)<\infty$, we could see
that $\left\langle \widehat{l},\chi\right\rangle $ can be defined by
approximation as a square integrable variable and $\mu(\left\langle
\widehat{l},\chi\right\rangle ^{2})=\int(G(x,y)^{2}\chi(dx)\chi(dy)$.

\bigskip

Let $z$ be a complex number such that $\operatorname{Re}(z)>0$.

Note also that$\ e^{-z\left\langle \widehat{l},f\right\rangle }+z\left\langle
\widehat{l},f\right\rangle -1$ is bounded by $\frac{\left|  z\right|  ^{2}}%
{2}\left\langle \widehat{l},f\right\rangle ^{2}$ and expands as an alternating
series $\sum_{2}^{\infty}\frac{z^{n}}{n!}(-\left\langle \widehat
{l},f\right\rangle )^{n}$, with $\left|  e^{-z\left\langle \widehat
{l},f\right\rangle }-1-\sum_{1}^{N}\frac{z^{n}}{n!}(-\left\langle \widehat
{l},f\right\rangle )^{n}\right|  \leq\frac{\left|  z\left\langle \widehat
{l},f\right\rangle \right|  ^{N+1}}{(N+1)!}.$ Then, for $\left|  z\right|  $
small enough., it follows from the above lemma that
\[
\mu(e^{-z\left\langle \widehat{l},f\right\rangle }+z\left\langle \widehat
{l},f\right\rangle -1)=\sum_{2}^{\infty}\frac{z^{n}}{n}Tr(-(M_{\sqrt{f}%
}GM_{\sqrt{f}})^{n})
\]
As $M_{\sqrt{f}}GM_{\sqrt{f}}$ is Hilbert-Schmidt $\det_{2}(I+zM_{\sqrt{f}%
}GM_{\sqrt{f}})$ is well defined and the second member writes -$\log(\det
_{2}(I+zM_{\sqrt{f}}GM_{\sqrt{f}}))$.\newline Then the identity%
\[
\mu(e^{-z\left\langle \widehat{l},f\right\rangle }+z\left\langle \widehat
{l},f\right\rangle -1)=-\log(\det{}_{2}(I+zM_{\sqrt{f}}GM_{\sqrt{f}})).
\]
extends, as both sides are analytic as locally uniform limits of analytic
functions, to all complex values with positive real part.

\bigskip

The renormalized occupation field $\ \widetilde{\mathcal{L}_{\alpha}}$ is
defined as the compensated sum of all $\widehat{l}$ in $\mathcal{L}_{\alpha}$
(formally, $\ \widetilde{\mathcal{L}_{\alpha}}=\widehat{\mathcal{L}_{\alpha}%
}-\int\int_{0}^{T(l)}\delta_{l_{s}}ds\mu(dl))$ $\ $By a standard argument used
for the construction of Levy processes,
\[
\left\langle \widetilde{\mathcal{L}_{\alpha}},f\right\rangle =\lim
_{\varepsilon\rightarrow0}(\sum_{\gamma\in\mathcal{L}_{\alpha}}%
(1_{\{T>\varepsilon\}}\int_{0}^{T}f(\gamma_{s})ds)-\alpha\mu
(1_{\{T>\varepsilon\}}\int_{0}^{T}f(\gamma_{s})ds))
\]
(we can denote $\lim_{\varepsilon\rightarrow0}\left\langle \widetilde
{\mathcal{L}_{\alpha,\varepsilon}},f\right\rangle )$  which converges a.s.
and in $L^{2}$, as
$$\mathbb{E}((\sum_{\gamma\in\mathcal{L}_{\alpha}}(1_{\{T>\varepsilon\}}%
\int_{0}^{T}f(\gamma_{s})ds)-\alpha\mu(1_{\{T>\varepsilon\}}\int_{0}%
^{T}f(\gamma_{s})ds))^{2})=\alpha\int(1_{\{T>\varepsilon\}}\int_{0}%
^{T}f(\gamma_{s})ds)^{2}\mu(dl)$$

and $\mathbb{E}(\left\langle \widetilde{\mathcal{L}_{\alpha}},f\right\rangle
^{2})=Tr((M_{\sqrt{f}}GM_{\sqrt{f}})^{2})$. Note that if we fix $f$, $\alpha$
can be considered as a time parameter and $\left\langle \widetilde
{\mathcal{L}_{\alpha,\varepsilon}},f\right\rangle $ as Levy processes with
discrete positive jumps approximating a Levy process with positive jumps
$\left\langle \widetilde{\mathcal{L}_{\alpha}},f\right\rangle $. The Levy
exponent $\mu(1_{\{T>\varepsilon\}}(e^{-\left\langle \widehat{l}%
,f\right\rangle }+\left\langle \widehat{l},f\right\rangle -1))$ of
$\left\langle \widetilde{\mathcal{L}_{\alpha,\varepsilon}},f\right\rangle )$
converges towards the L\'{e}vy exponent of $\left\langle \widetilde
{\mathcal{L}_{\alpha}},f\right\rangle )$ which is $\mu((e^{-\left\langle
\widehat{l},f\right\rangle }+\left\langle \widehat{l},f\right\rangle
-1))$.\newline and, from the identity $E(e^{-\left\langle \widetilde
{\mathcal{L}_{\alpha}},f\right\rangle })=e^{-\alpha\mu(e^{-\left\langle
\widehat{l},f\right\rangle }+\left\langle \widehat{l},f\right\rangle -1)}$, we
get the

\begin{theorem}
Assume $d\leq3$. Denoting $\widetilde{\mathcal{L}_{\alpha}}$ the compensated
sum of all $\widehat{l}$ in $\mathcal{L}_{\alpha}$, we have $\mathbb{E}%
(e^{-\left\langle \widetilde{\mathcal{L}_{\alpha}},f\right\rangle })=\det
_{2}(I+M_{\sqrt{f}}GM_{\sqrt{f}}))^{-\alpha}$
\end{theorem}

Moreover $e^{-\left\langle \widetilde{\mathcal{L}_{\alpha,\varepsilon}%
},f\right\rangle }$ converges a.s. and in $L^{1}$\ towards $e^{-\left\langle
\widetilde{\mathcal{L}_{\alpha}},f\right\rangle }$.\newline
Considering distributions of finite energy $\chi$ (i.e. such that
${\int(G(x,y)^{2}\chi(dx)\chi(dy)<\infty}$), we can see that $\left\langle
\widetilde{\mathcal{L}_{\alpha}},\chi\right\rangle $ can be defined by
approximation as $\lim_{\lambda\rightarrow\infty}(\left\langle \widetilde
{\mathcal{L}_{\alpha}},\lambda G_{\lambda}\chi\right\rangle )$ and
\[
\mathbb{E}(\left\langle \widetilde{\mathcal{L}_{\alpha}},\chi\right\rangle
^{2})=\alpha\int(G(x,y))^{2}\chi(dx)\chi(dy).
\]
Specializing to $\alpha=\frac{k}{2}$, $k$ being any positive integer we have:

\begin{corollary}
The renormalized occupation field $\widetilde{\mathcal{L}_{\frac{k}{2}}}$ and
the Wick square $\frac{1}{2}:\sum_{1}^{k}\phi_{l}^{2}:$ have the same distribution.
\end{corollary}

\bigskip

If $\Theta$ is a conformal map from $D$ onto $\Theta(D)$, it follows from the
conformal invariance of the Brownian trajectories that a similar property
holds for the bBrownian''loop soup''(Cf \cite{LW}). More precisely, if
$c(x)=Jacobian_{x}(\Theta)$ and, given a loop $l$, if $T^{c}(l)$ denotes the
reparametrized loop $l_{\tau_{s}}$, with $\int_{0}^{\tau_{s}}c(l_{u})du=s$,
$\Theta T^{c}(\mathcal{L}_{\alpha})$ is the Brownian loop soup of intensity
parameter $\alpha$ on $\Theta(D)$. Then we have the following:

\begin{proposition}
$\Theta(c\widetilde{\mathcal{L}_{\alpha}})$ is the renormalized occupation
field on $\Theta(D)$.
\end{proposition}

\begin{proof}
\bigskip We have to show that the compensated sum is the same if we perform it
after or before the time change. For this it is enough to check that
\begin{align*}
&  \ \mathbb{E}([\sum_{\gamma\in\mathcal{L}_{\alpha}}(1_{\{\tau_{T}>\eta
\}}1_{\{T\leq\varepsilon\}}\int_{0}^{T}f(\gamma_{s})ds-\alpha\int
(1_{\{\tau_{T}>\eta\}}1_{\{T\leq\varepsilon\}}\int_{0}^{T}f(\gamma_{s}%
)ds)\mu(d\gamma)]^{2})\\
&  =\alpha\int(1_{\{\tau_{T}>\eta\}}1_{\{T\leq\varepsilon\}}\int_{0}%
^{T}f(\gamma_{s})ds)^{2}\mu(d\gamma)
\end{align*}
and%
\begin{align*}
&  \mathbb{E}([\sum_{\gamma\in\mathcal{L}_{\alpha}}(1_{\{T>\varepsilon
\}}1_{\tau_{T}\leq\eta}\int_{0}^{T}f(\gamma_{s})ds-\alpha\int
(1_{\{T>\varepsilon\}}1_{\tau_{T}\leq\eta}\int_{0}^{T}f(\gamma_{s}%
)ds)\mu(d\gamma)]^{2})\\
&  \alpha\int(1_{\{T>\varepsilon\}}1_{\tau_{T}\leq\eta}\int_{0}^{T}%
f(\gamma_{s})ds)^{2}\mu(d\gamma)\
\end{align*}
converge to zero as $\varepsilon$ and $\eta$ go to zero. It follows from the
fact that:
\[
\int[1_{\{T\leq\varepsilon\}}\int_{0}^{T}f(\gamma_{s})ds]^{2}\mu(d\gamma)
\]
and
\[
\int[1_{\tau_{T}\leq\eta}\int_{0}^{T}f(\gamma_{s})ds]^{2}\mu(d\gamma)
\]
converge to $0$. The second follows easily from the first\ if $c$ is bounded
away from zero. We can always consider the ''loop soups'' in an increasing
sequence of relatively compact open subsets of $D$ to reduce the general case
to that situation.
\end{proof}

As in the discrete case (see corollary \ref{mom}), we can compute product
expectations. In dimensions one and two, for $f_{j}$ continuous functions with
compact support in $D$:
\begin{equation}
\mathbb{E(}\left\langle \widetilde{\mathcal{L}_{\alpha}},f_{1}\right\rangle
...\left\langle \widetilde{\mathcal{L}_{\alpha}},f_{k}\right\rangle )=\int
Per_{\alpha}^{0}(G(x_{l},x_{m}),1\leq l,m\leq k)\prod f_{j}(x_{j})dx_{j}
\label{prodl2}%
\end{equation}

\section{Renormalized powers\bigskip}

In dimension one, powers of the occupation field can be viewed as integrated
self intersection local times. In dimension two, renormalized powers of the
occupation field, also called \textsl{renormalized self intersections local
times }can be defined as follows:

\begin{theorem}
Assume $d=2$. Let $\pi_{\varepsilon}^{x}(dy)$ be the normalized arclength on
the circle of radius $\varepsilon$ around $x$, and set $\sigma_{\varepsilon
}^{x}=\int G(y,z)\pi_{\varepsilon}^{x}(dy)\pi_{\varepsilon}^{x}(dz)$. Then,
$\int f(x)Q_{k}^{\alpha,\sigma_{\varepsilon}^{x}}(\left\langle \widetilde
{\mathcal{L}_{\alpha}},\pi_{\varepsilon}^{x}\right\rangle )dx$ converges in
$L^{2}$ for any bounded continuous function $f$ with compact support towards a
limit denoted $\left\langle \widetilde{\mathcal{L}_{\alpha}^{k}}%
,f\right\rangle $ and

$\mathbb{E}(\left\langle \widetilde{\mathcal{L}_{\alpha}^{k}},f\right\rangle
\left\langle \widetilde{\mathcal{L}_{\alpha}^{l}},h\right\rangle
)=\delta_{l,k}\frac{\alpha(\alpha+1)...(\alpha+k-1)}{k!}\int G^{2k}(x,y)f(x)h(y)dxdy.$
\end{theorem}

\begin{proof}
The idea of the proof can be understood by trying to prove that

$\mathbb{E((}\int f(x)Q_{k}^{\alpha,\sigma_{x}^{\varepsilon}}(\left\langle
\widetilde{\mathcal{L}_{\alpha}},\pi_{\varepsilon}^{x}\right\rangle )dx)^{2}%
)$\ remains bounded as $\varepsilon$ decreses to zero. The idea is to expand
this expression in terms of sums of integrals of product of Green functions
and check that the combinatorial identities (\ref{null})\ imply the
cancelation of the logarithmic divergences.

This is done by showing (as done below in the proof of the theorem) one can
modify slightly the products of Green functions appearing in $\mathbb{E(}%
Q_{k}^{\alpha,\sigma_{\varepsilon}^{x}}(\left\langle \widetilde{\mathcal{L}%
_{\alpha}},\pi_{\varepsilon}^{x}\right\rangle )Q_{k}^{\alpha,\sigma
_{\varepsilon}^{y}}(\left\langle \widetilde{\mathcal{L}_{\alpha}}%
,\pi_{\varepsilon}^{y}\right\rangle ))$\ to replace them by products of the
form $G(x,y)^{j}(\sigma_{\varepsilon}^{x})^{l}\sigma_{\varepsilon}^{y})^{h}%
$\ . The cancelation of terms containing $\sigma_{\varepsilon}^{x}$ and/or
$\sigma_{\varepsilon}^{y}$ then follows directly from the combinatorial indentities.

Let us now prove the theorem. Consider first, for any $x_{1,}x_{2}...x_{n}$,
$\varepsilon$ small enough and $\varepsilon\leq\varepsilon_{1},...\varepsilon
_{n}\leq2\varepsilon$, with $\varepsilon_{i}=\varepsilon_{j}$ if $x_{i}=x_{j}%
$, an expression of the form:
\[
\Delta=\left|  \prod_{i,x_{i-1}\neq x_{i}}G(x_{i-1},x_{i})(\sigma
_{\varepsilon_{i}}^{x_{i}})^{m_{i}}-\int G(y_{1},y_{2})...G(y_{n},y_{1}%
)\pi_{\varepsilon_{1}}^{x_{1}}(dy_{1})...\pi_{\varepsilon_{n}}^{x_{n}}%
(dy_{n})\right|
\]
in which we define $m_{i}$ as $\sup(h,\!x_{i+h}=x_{i})$.\newline In the
integral term, we first replace progressively $G(y_{i-1},y_{i})$ by
$G(x_{i-1},x_{i})$ whenever $x_{i-1}\neq x_{i}$, using triangle, then Schwartz
inequality, to get an upper bound of the absolute value of the difference made
by this substitution in terms of a sum $\Delta^{\prime}$\ of expressions of
the form
\[
\prod_{l}G(x_{l},x_{l+1})\sqrt{\int(G(y_{1},y_{2})-G(x_{1},x_{2}))^{2}%
\pi_{\varepsilon_{1}}^{x_{1}}(dy_{1})\pi_{\varepsilon_{2}}^{x_{2}}(dy_{2}%
)\!\!\int\!\prod G^{2}(y_{k},y_{k+1})\prod\pi_{\varepsilon_{k}}^{x_{k}}(dy_{k})}.
\]
The expression obtained after these substitutions can be written
\[
W=\prod_{i,x_{i-1}\neq x_{i}}G(x_{i-1},x_{i})\int G(y_{1},y_{2}%
)...G(y_{m_{i-1}},y_{m_{i}})\pi_{\varepsilon_{i}}^{x_{i}}(dy_{1}%
)...\pi_{\varepsilon_{i}}^{x_{i}}(dy_{m_{i}})
\]
and we see the integral terms could be replaced by $(\sigma_{\varepsilon
}^{x_{i}})^{m_{i}}$ if $G$ was translation invariant. But as the distance
between $x$ and $y$ tends to $0$, $G(x,y)$ is equivalent to $G_{0}%
(x,y)=\frac{1}{\pi}\log(\left\|  x-y\right\|  )$ and moreover, $G(x,y)=G_{0}%
(x,y)-H^{D^{c}}(x,dz)G_{0}(z,y)$, $H^{D^{c}}$ denoting the Poisson kernel on
the boundary of $D$. As our points lie in a compact inside $D$, it follows
that for some constant $C$, for $\left\|  y_{1}-x\right\|  \leq\varepsilon$,
$\left|  \int(G(y_{1},y_{2})\pi_{\varepsilon}^{x}(dy_{2})-\sigma_{\varepsilon
}^{x}\right|  <C\varepsilon$. \newline Hence, the difference $\Delta
^{\prime\prime}$ between $W$ and $\prod_{i,x_{i-1}\neq x_{i}}G(x_{i-1}%
,x_{i})(\sigma_{\varepsilon}^{x_{i}})^{m_{i}}$ can be bounded by $\varepsilon
W^{\prime}$, where $W^{\prime}$ is an expression similar to $W$..

To get a good upper bound on $\Delta$, using the previous observations, by
repeated applications of H\"{o}lder inequality. it is enough to show that for
$\varepsilon$ small enough , $C$ and $C^{\prime}$\ denoting various constants:

\begin{enumerate}
\item[1)] $\int(G(y_{1},y_{2})-G(x_{1},x_{2})^{2}\pi_{\varepsilon_{1}}^{x_{1}%
}(dy_{1})\pi_{\varepsilon_{2}}^{x_{2}}(dy_{2})$\newline $< C(\varepsilon
1_{\{\left\|  x_{1}-x_{2}\right\|  \geq\sqrt{\varepsilon}\}}+(G(x_{1}%
,x_{2})^{2}+\log(\varepsilon)^{2})1_{\{\left\|  x_{1}-x_{2}\right\|
<\sqrt{\varepsilon}\}})$

\item[2)] $\int G(y_{1},y_{2})^{k}\pi_{\varepsilon}^{x}(dy_{1}) \pi
_{\varepsilon}^{x}(dy_{2})<C\left|  \log(\varepsilon)\right|  ^{k}$

\item[3)] $\int G(y_{1},y_{2})^{k}\pi_{\varepsilon_{1}}^{x_{1}}(dy_{1}%
)\pi_{\varepsilon_{2}}^{x_{2}}(dy_{2})<C\left|  \log(\varepsilon)\right|  ^{k}$
\end{enumerate}

As the main contributions come from the singularities of $G$, they follow from
the following simple inequalities:

\begin{enumerate}
\item[1')]
\begin{multline*}
\int\left|  \log(\varepsilon^{2}+2R\varepsilon\cos(\theta)+R^{2}%
)-\log(R)\right|  ^{2}d\theta\\
=\int\left|  \log((\varepsilon/R)^{2}+2(\varepsilon/R)\cos(\theta)+1)\right|
^{2}d\theta<C((\varepsilon1_{\{R\geq\sqrt{\varepsilon}\}\}}+\log
^{2}(R/\varepsilon)1_{\{R<\sqrt{\varepsilon}\}\}})
\end{multline*}
(considering separately the cases where $\frac{\varepsilon}{R}$ is large or small)

\item[2')] $\int\left|  \log(\varepsilon^{2}(2+2\cos(\theta)))\right|
^{k}d\theta\leq C\left|  \log(\varepsilon)\right|  ^{k}$

\item[3')] $\!\int\!\left|\log(\varepsilon_{1}\cos(\theta_{1})+\varepsilon
_{2}\cos(\theta_{2})+r)^{2}+(\varepsilon_{1}\sin(\theta_{1})+\varepsilon
_{2}\sin(\theta_{2}))^{2}\right|^{k}d\theta_{1}d\theta_{2}\leq\! C(\left|
\log(\varepsilon)\right|)^{k}$. It can be proved by observing that for
$r\leq\varepsilon_{1}+\varepsilon_{2}$, we have near the singularities (i.e.
the values $\theta_{1}(r)$ and $\theta_{2}(r)$ for which the expression under
the $\log$\ vanishes) to evaluate integrals bounded by $C\int_{0}^{1}%
(-\log(\varepsilon u))^{k}du\leq C^{\prime}(-\log(\varepsilon))^{k}$ for
$\varepsilon$ small enough.
\end{enumerate}

Let us now show that for $\varepsilon\leq\varepsilon_{1},\varepsilon_{2}%
\leq2\varepsilon$, we have, for some integer $N^{n,k}$%

\begin{multline}
\left|  \mathbb{E}(Q_{k}^{\alpha,\sigma_{x}^{\varepsilon_{1}}}(\left\langle
\widetilde{\mathcal{L}_{\alpha}},\pi_{\varepsilon_{1}}^{x}\right\rangle
)Q_{l}^{\alpha,\sigma_{y}^{\varepsilon_{2}}}(\left\langle \widetilde
{\mathcal{L}_{\alpha}},\pi_{\varepsilon_{2}}^{y}\right\rangle ))-\delta
_{l,k}G(x,y)^{2k}\frac{\alpha(\alpha+1)...(\alpha+k-1)}{k!})\right| \\
\leq C\log(\varepsilon)^{N_{l,k}}(\sqrt{\varepsilon}+G(x,y)^{2k}1_{\{\left\|
x-y\right\|  <\sqrt{\varepsilon}}) \label{maj}%
\end{multline}
Indeed, developing the polynomials and using formula (\ref{prodl2}) we can
express this expectation as a linear combination of integrals under
$\displaystyle                        {\prod_{i}\pi_{\varepsilon_{1}}%
^{x}(dx_{i})\prod_{j}\pi_{\varepsilon_{2}}^{y}(dy_{j})}$ of products of
$G(x_{i},y_{i^{\prime}}),G(x_{i},x_{j})$ and $G(y_{j},y_{j^{\prime}})$ as we
did in the discrete case. If we replace each $G(x_{i},y_{j})$ by $G(x,y)$,
each $G(x_{i},x_{i^{\prime}})$ by $\sigma_{\varepsilon_{1}}^{x}$ and each
$G(y_{j},y_{j^{\prime}})$ by $\sigma_{\varepsilon_{2}}^{y}$, we can use the
combinatorial identity (\ref{null})\ to get the value $\displaystyle
{\delta_{l,k}G(x,y)^{2k}\frac{\alpha(\alpha+1)...(\alpha+k-1)}{k!}}$. Then,
the above results allow to bound the error made by this replacement.

The bound (\ref{maj}) is uniform in $(x,y)$ only away from the diagonal as
$G(x,y)$ can be arbitrarily large but we conclude from it\ that for any
bounded integrable $f$ and $h$,
\begin{multline*}
\left|\int\!(\mathbb{E}(Q_{k}^{\alpha,\sigma_{x}^{\varepsilon_{1}}%
}\!(\langle \widetilde{\mathcal{L}_{\alpha}},\pi_{\varepsilon_{1}}%
^{x}\rangle)Q_{l}^{\alpha,\sigma_{y}^{\varepsilon_{2}}}\!(\langle
\widetilde{\mathcal{L}_{\alpha}},\pi_{\varepsilon_{2}}^{y}\rangle
))-\delta_{l,k}G(x,y)^{2k}\frac{\alpha\cdots(\alpha+k-1)}{k!})f(x)h(y)dxdy\right|
\\
\leq C^{\prime}\sqrt{\varepsilon}\log(\varepsilon)^{N_{l,k}}%
\end{multline*}
(as $\int\int G(x,y)^{2k}1_{\{\left\|  x-y\right\|  <\sqrt{\varepsilon}}dxdy$
can be bounded by $C\varepsilon^{\frac{2}{3}}$, for example).

Taking $\varepsilon_{n}=2^{-n}$, it is then straightforward to check that
$\int f(x)Q_{k}^{\alpha,\sigma_{x}^{\varepsilon_{n}}}(\langle
\widetilde{\mathcal{L}_{\alpha}},\pi_{\varepsilon_{n}}^{x}\rangle
)dx$\ is a Cauchy sequence in $L^{2}$. The theorem follows.
\end{proof}

\bigskip

Specializing to $\alpha=\frac{k}{2}$, $k$ being any positive integer\ as
before, Wick powers of $\sum_{j=1}^{k}\phi_{j}^{2}$\ are associated with self
intersection local times of the loops. More precisely, we have:

\begin{proposition}
The renormalized self intersection local times $\widetilde{\mathcal{L}%
_{\frac{k}{2}}^{n}}$ and the Wick powers $\frac{1}{2^{n}n!}:(\sum_{1}^{k}%
\phi_{l}^{2})^{n}:$ have the same joint distribution.
\end{proposition}

The proof is similar to the one given in \cite{LJ1} and also to the proof of
the above theorem, but simpler. It is just a calculation of the $L^{2}$-norm
of
\[
\int[:(\phi^{2})^{n}:(x)-Q_{n}^{\frac{1}{2},\sigma_{x}^{\varepsilon}}%
(:\phi_{x}^{2}:(\pi_{\varepsilon}^{x}))]f(x)dx
\]
which converges to zero with $\varepsilon$.

\subsubsection*{Final remarks:}

\begin{enumerate}
\item[a)] These generalized fields have two fundamental properties:

Firstly they are local fields (or more precisely local functionals of the
field $\widetilde{\mathcal{L}_{\alpha}}$ in the sense that their values on
functions supported in an open set $D$ depend only on the trace of the loops
on $D$.

Secondly, noting we could use different regularizations to define
$\widetilde{\mathcal{L}_{\alpha}^{k}}$, the action of a conformal
transformation $\Theta$\ on these fields is given by \emph{the }$k$\emph{-th
power of the conformal factor }$c=Jacobian(\Theta)$. More precisely,
$\Theta(c^{k}\widetilde{\mathcal{L}_{\alpha}^{k}})$ is the renormalized $k$-th
power of the occupation field in $\Theta(D)$.

\item[b)] It should be possible to derive from the above remark the existence
of exponential moments and introduce non trivial local interactions as in the
constructive field theory derived from the free field (Cf \cite{Sim2}).

\item[c)] Let us also briefly consider currents. We will restrict our
attention to the one and two dimensional Brownian case, $X$ being an open
subset of the line or plane. Currents can be defined by vector fields, with
compact support.

Then, if now we denote by $\phi$ the complex valued free field (its real and
imaginary parts being two independent copies of the free field), $\int
_{l}\omega$ and $\int_{X}(\overline{\phi}\partial_{\omega}\phi-\phi
\partial_{\omega}\overline{\phi})dx$ are well defined square integrable
variables in dimension 1 (it can be checked easily by Fourier series). The
distribution of the centered occupation field of the loop process ''twisted''
by the complex exponential $\exp(\sum_{l\in\mathcal{L}_{\alpha}}\int
_{l}i\omega+\frac{1}{2}\widehat{l}(\left\|  \omega\right\|  ^{2}))$ appears to
be the same as the distribution of the field $:\phi\overline{\phi}%
:$\ ''twisted'' by the complex exponential $\exp(\int_{X}(\overline{\phi
}\partial_{\omega}\phi-\phi\partial_{\omega}\overline{\phi})dx)$ (Cf\cite{LJ2}).

In dimension 2, logarithmic divergences occur.

\item[d)] There is a lot of related investigations. The extension of the
properties proved here in the finite framework has still to be completed,
though the relation with spanning trees should follow from the remarkable
results obtained on SLE processes, especially \cite{LSW}. Note finally that
other essential relations between SLE processes, loops and free fields appear
in \cite{WW}, \cite{ScSh}\ and \cite{Dub}.
\end{enumerate}

\end{document}